\documentclass{amsart}

\usepackage[foot]{amsaddr}
\let\oldtocsection=\tocsection

\let\oldtocsubsection=\tocsubsection 

\let\oldtocsubsubsection=\tocsubsubsection

\renewcommand{\tocsection}[2]{\vspace{0.5em}\hspace{0em}\oldtocsection{#1}{#2}}
\renewcommand{\tocsubsection}[2]{\vspace{0.5em}\hspace{1em}\oldtocsubsection{#1}{#2}}
\renewcommand{\tocsubsubsection}[2]{\vspace{0.5em}\hspace{2em}\oldtocsubsubsection{#1}{#2}}
\usepackage{graphicx,cancel,xcolor,hyperref,comment,graphicx,geometry}

\usepackage{tikz}
\usetikzlibrary{decorations.pathreplacing}
\usepackage{graphicx,url,etoolbox}
\usepackage{lipsum}
\usepackage{dsfont}
\usepackage{amsmath}

\usepackage{fancyhdr}
\usepackage{lastpage}

\def\nline{\\ \noalign{\medskip}}

\newtheorem{theoreme}{Theorem}[section]

\theoremstyle{definition}

\numberwithin{equation}{section}

\renewenvironment{proof}{{\bfseries \noindent Proof.}}{\demo}
\newcommand\xqed[1]{%
	\leavevmode\unskip\penalty9999 \hbox{}\nobreak\hfill
	\quad\hbox{#1}}
\newcommand\demo{\xqed{$\square$}}

\hypersetup{bookmarks, colorlinks, urlcolor=blue, citecolor=blue, linkcolor=blue, hyperfigures, pagebackref,
	pdfcreator=LaTeX, breaklinks=true, pdfpagelayout=SinglePage, bookmarksopen=true,bookmarksopenlevel=2}

\usepackage{comment}

\def\R{\mathbb R}

\def\HH{\mathcal H}

\def\AA{\mathcal A}

\def\la {{\lambda}}

\newcommand {\nc}   {\newcommand}
\nc {\be}   {\begin{equation}} \nc {\ee}   {\end{equation}} \nc
{\beq}  {\begin{eqnarray}} \nc {\eeq}  {\end{eqnarray}} \nc {\beqs}
{\begin{eqnarray*}} \nc {\eeqs} {\end{eqnarray*}}
\def\edc{\end{document}}

\providecommand{\abs}[1]{\lvert#1\rvert}

\numberwithin{equation}{section}

\theoremstyle{Thm}
\newtheorem{Thm}{Theorem}[section]
\newtheorem{lem}{Lemma}[section]
\newtheorem{prop}{Proposition}[section]

\newtheorem{rk}{Remark}[section]
\definecolor{carnelian}{rgb}{0.7, 0.11, 0.11}
\definecolor{carmine}{rgb}{0.59, 0.0, 0.09}
\definecolor{burgundy}{rgb}{0.5, 0.0, 0.13}
\definecolor{darkmidnightblue}{rgb}{0.0, 0.2, 0.4}
\definecolor{dimgray}{rgb}{0.75, 0.75, 0.75}
\definecolor{palecarmine}{rgb}{0.69, 0.25, 0.21}
\newcounter{dummy} 
\numberwithin{dummy}{section}
\newtheorem{Theorem}[dummy]{Theorem}

\newtheorem{defi}[dummy]{Definition}

\newtheorem{Remark}[dummy]{Remark}

\numberwithin{equation}{section}

\def\AA{\mathcal A}
\def\HH{\mathbf{\mathcal H}}

\newcommand{\intdx}{\int_{0}^{L} }

\newcommand{\intdnb}{\int_{0 }^{\beta } }

\newcommand{\intdnc}{\int_{\alpha }^{\gamma } }

\newcommand{\intdni}{\int_{\alpha }^{\beta } }

\providecommand{\abs}[1]{\lvert#1\rvert}
\begin{document}
	\title[\fontsize{7}{9}\selectfont  ]{Stability results of a singular local interaction elastic/viscoelastic coupled wave equations with time delay}
\author{Mohammad Akil$^{1}$}
\author{Haidar Badawi$^{2}$}
\author{Ali Wehbe$^{3}$}
\address{$^1$ Universit\'e Savoie Mont Blanc, Laboratoire LAMA, Chamb\'ery-France}
\address{$^2$ Universit\'e Polytechnique Hauts-de-France
Valenciennes, LAMAV, France}
\address{$^3$Lebanese University, Faculty of sciences 1, Khawarizmi Laboratory of  Mathematics and Applications-KALMA, Hadath-Beirut, Lebanon.}
\email{mohamadakil1@hotmail.com, Haidar.Badawi@etu.uphf.fr,   \\ ali.wehbe@ul.edu.lb}
\begin{abstract}
The purpose of this paper is to investigate the stabilization of a one-dimensional coupled wave equations with non smooth localized viscoelastic damping of Kelvin-Voigt type and localized time delay. Using a general criteria of Arendt-Batty, we show the strong stability of our system in the absence of the compactness of the resolvent. Finally, using frequency domain approach combining with a multiplier method, we prove a polynomial energy decay rate of order  $t^{-1}$.
\end{abstract}

\maketitle
\pagenumbering{roman}
\maketitle
\tableofcontents
\pagenumbering{arabic}
\setcounter{page}{1}
\newpage
\section{Introduction}\noindent 
\subsection{Description of the paper}
In this paper, we investigate the stability of local coupled  wave equations with singular localized viscoelastic damping of Kelvin-Voigt type and localized  time delay. More precisely, we consider the following System: 
\begin{equation}\label{sysorig} \left\{	\begin{array}{llll}\vspace{0.15cm}
		u_{tt}-\left[au_x +b(x)(\kappa_1 u_{tx}+\kappa_2 u_{tx}\left(x,t-\tau)\right)\right]_x +c(x)y_t =0,& (x,t)\in (0,L)\times (0,\infty) ,&\\  \vspace{0.15cm}
		y_{tt}-y_{xx}-c(x)u_t =0,  &(x,t)\in (0,L)\times (0,\infty) ,&\\\vspace{0.15cm}
		u(0,t)=u(L,t)=y(0,t)=y(L,t)=0,& t>0 ,& \\\vspace{0.15cm}
		(u(x,0),u_t (x,0))=(u_0 (x),u_1 (x)), &x\in (0,L),&\\\vspace{0.15cm}	(y(x,0),y_t (x,0))=(y_0 (x),y_1 (x)), &x\in (0,L),&\\\vspace{0.15cm}
		u_t (x,t)=f_0 (x,t), &(x,t)\in(0,L)\times(-\tau,0),&
	\end{array}\right.
\end{equation}
where $L,\tau,a$ and $\kappa_1$ are  positive real numbers, $\kappa_2$ is a non-zero real number and $(u_0,u_1,y_0,y_1,f_0)$ belongs to a suitable space. We suppose that there exists $0<\alpha <\beta <\gamma  <L$ and a  positive constant $c_0$, such that   
\begin{equation*}
	b(x)=\left\{	\begin{array}{lll}\vspace{0.15cm}
		1,& x\in (0 ,\beta), &\\
		0,&x\in (\beta ,L),&
	\end{array}\right. \text{and }\quad
	c(x)=\left\{	\begin{array}{lll}\vspace{0.15cm}
		c_0 ,& x\in (\alpha ,\gamma), &\\
		0,&x\in (0,\alpha )\cup (\gamma ,L).&
	\end{array}\right.
\end{equation*}
The Figure \ref{FIG-1} describes system \eqref{sysorig}
\begin{figure}[h!]
	\begin{center}
		\begin{tikzpicture}
		

		
		
		\draw[-,dimgray][thick](4.125,0)--(4.125,2);
		\draw[-,dimgray][thick](4.25,0)--(4.25,2);
		\draw[-,dimgray][thick](4.375,0)--(4.375,2);
		\draw[-,dimgray][thick](4.5,0)--(4.5,2);
		\draw[-,dimgray][thick](4.625,0)--(4.625,2);
		\draw[-,dimgray][thick](4.75,0)--(4.75,2);
		\draw[-,dimgray][thick](4.875,0)--(4.875,2);
		\draw[-,dimgray][thick](5,0)--(5,2);
		\draw[-,dimgray][thick](5.125,0)--(5.125,2);
		\draw[-,dimgray][thick](5.25,0)--(5.25,2);
		\draw[-,dimgray][thick](5.375,0)--(5.375,2);
		\draw[-,dimgray][thick](5.5,0)--(5.5,2);
		\draw[-,dimgray][thick](5.625,0)--(5.625,2);
		\draw[-,dimgray][thick](5.75,0)--(5.75,2);
		\draw[-,dimgray][thick](5.875,0)--(5.875,2);
		
		\draw[-,dimgray][thick](6.125,0)--(6.125,2);
		\draw[-,dimgray][thick](6.25,0)--(6.25,2);
		\draw[-,dimgray][thick](6.375,0)--(6.375,2);
		\draw[-,dimgray][thick](6.5,0)--(6.5,2);
		\draw[-,dimgray][thick](6.625,0)--(6.625,2);
		\draw[-,dimgray][thick](6.75,0)--(6.75,2);
		\draw[-,dimgray][thick](6.875,0)--(6.875,2);
		\draw[-,dimgray][thick](7,0)--(7,2);
		\draw[-,dimgray][thick](7.125,0)--(7.125,2);
		\draw[-,dimgray][thick](7.25,0)--(7.25,2);
		\draw[-,dimgray][thick](7.375,0)--(7.375,2);
		\draw[-,dimgray][thick](7.5,0)--(7.5,2);
		\draw[-,dimgray][thick](7.625,0)--(7.625,2);
		\draw[-,dimgray][thick](7.75,0)--(7.75,2);
		\draw[-,dimgray][thick](7.875,0)--(7.875,2);

		\draw[-,palecarmine][thick](2,0.125)--(6,0.125);
		\draw[-,palecarmine][thick](2,0.25)--(6,0.25);
		\draw[-,palecarmine][thick](2,0.375)--(6,0.375);
		\draw[-,palecarmine][thick](2,0.5)--(6,0.5);
		\draw[-,palecarmine][thick](2,0.625)--(6,0.625);
		\draw[-,palecarmine][thick](2,0.75)--(6,0.75);
		\draw[-,palecarmine][thick](2,0.875)--(6,0.875);
		\draw[-,palecarmine][thick](2,1)--(6,1);
		\draw[-,palecarmine][thick](2,1.125)--(6,1.125);
		\draw[-,palecarmine][thick](2,1.25)--(6,1.25);
		\draw[-,palecarmine][thick](2,1.375)--(6,1.375);
		\draw[-,palecarmine][thick](2,1.5)--(6,1.5);
		\draw[-,palecarmine][thick](2,1.625)--(6,1.625);
		\draw[-,palecarmine][thick](2,1.75)--(6,1.75);
		\draw[-,palecarmine][thick](2,1.875)--(6,1.875);


		\node[darkmidnightblue,above] at (3.85,2.75){\scalebox{0.75}{ Viscoelastic region \& time delay  }}; 
		\node[darkmidnightblue,below] at (6,-0.75){\scalebox{0.75}{ Coupling region}}; 

		\draw[-,darkmidnightblue][ultra thick](2,0)--(10,0);
		\draw[-,darkmidnightblue][ultra thick](2,2)--(10,2);

		\draw[-,darkmidnightblue][ultra thick](10,0)--(10,2);
		\draw[-,darkmidnightblue][ultra thick](2,0)--(2,2);
		\draw[-,darkmidnightblue][ultra thick](4,0)--(4,2);
		\draw[-,darkmidnightblue][ultra thick](6,0)--(6,2);
		\draw[-,darkmidnightblue][ultra thick](8,0)--(8,2);

		\node at (2,0) [circle, scale=0.4, draw=darkmidnightblue!80,fill=darkmidnightblue!80] {};
		\node[darkmidnightblue,below] at (2,0){\scalebox{0.8}{$0$}};   
		
		\node at (4,0) [circle, scale=0.4, draw=darkmidnightblue!80,fill=darkmidnightblue!80] {};
		\node[darkmidnightblue,below] at (4,0){\scalebox{0.8}{$\alpha$}};        
		
		\node at (6,0) [circle, scale=0.4, draw=darkmidnightblue!80,fill=darkmidnightblue!80] {};
		\node[darkmidnightblue,below] at (6,0){\scalebox{0.8}{$\beta$}};   
		
		\node at (8,0) [circle, scale=0.4, draw=darkmidnightblue!80,fill=darkmidnightblue!80] {};
		\node[darkmidnightblue,below] at (8,0){\scalebox{0.8}{$\gamma$}};  
		
		\node at (10,0) [circle, scale=0.4, draw=darkmidnightblue!80,fill=darkmidnightblue!80] {};
		\node[darkmidnightblue,below] at (10,0){\scalebox{0.8}{$L$}}; 
		
		
		
		\draw [darkmidnightblue,decorate,decoration={brace,amplitude=10pt},xshift=-1pt,yshift=0pt]
		(2,2.25) -- (6,2.25)  ;	
		
		\draw [darkmidnightblue,decorate,decoration={brace,amplitude=10pt,mirror,raise=4pt},yshift=0pt]
		(4,-0.25) -- (8,-0.25)  ;

		\end{tikzpicture}
		\caption{Local Kelvin-Voigt damping and Local time delay feedback.}\label{FIG-1}
	\end{center}
\end{figure}
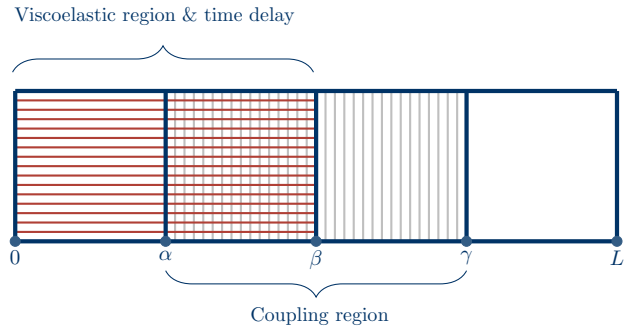\\

System \eqref{sysorig} consists of two wave equations with only one singular viscoelastic damping acting on the first equation, the second one is indirectly damped via a singular coupling between the two equations. The notion of indirect damping mechanisms has been introduced by Russell in \cite{Russell01} and since then, it has  attracted the attention of many authors (see for instance  \cite{akil2018influence}, \cite{Alabau04}, \cite{Alabau02}, \cite{Ammari02}, \cite{Cui2016}, \cite{Abdallah2019}, \cite{LiuRao01} and \cite{ZhangZuazua01} ). The study of such systems is also motivated by several physical considerations like Timoshenko and Bresse systems (see for instance \cite{ghaderbresse}, \cite{AKIL02}, \cite{Denisr01} and \cite{Najdibresse}). In fact, there are few results concerning the stability of coupled wave equations with local Kelvin-Voigt damping without time delay, especially in the absence of smoothness of the damping and coupling coefficients  (see Subsection \ref{subcs} ). The last motivates our interest to study the stabilization of system \eqref{sysorig} in the present paper.\\
 \subsection{Previous Literature}
 The wave is created when a vibrating source disturbs  the medium. In order to restrain those vibrations, several damping can be added such as Kelvin-Voigt damping which is originated from the extension or compression of the vibrating particles. This damping is a viscoelastic structure having properties of both elasticity and viscosity. In the recent years, many researchers showed interest in problems involving this kind of damping where different types of stability, depend on the smoothness of the damping coefficients,  has been showed (see \cite{Alves2014}, \cite{Alves2013}, \cite{F.HASSINE2015}, \cite{HASSINE201584}, \cite{Huang-1988}, \cite{chenLiuLiu-1998},  \cite{Liu2016},   \cite{Portillo2017} and \cite{rivera2018} ). However, time delays have been used in several applications such as in  physical, chemical, biological, thermal phenomenas not only depend on the present state but also on some past occurrences (see \cite{PhysRevE.57.2150},\cite{Kol-mish}) . In the last years, the control of partial differential equations with time delays have become popular among scientists. In many cases the time delay induce  some instabilities see \cite{datko1-inst,datko2-inst,datko1985,dreher2009}.\\\linebreak
However, let us  recall briefly some systems of wave equations with Kelvin-Voigt damping and time delay represented in previous literature.
\subsubsection{Coupled wave equations with Kelvin-Voigt damping and without time delay}\label{subcs}
 In 2019, Hassine and Souayeh in \cite{hassine2019stability} studied the behavior of a system with coupled wave equations with a partial Kelvin-Voigt damping, by considering the following system 
\begin{equation}\label{sys2019}
	\left\{\begin{array}{lll}
		u_{tt}-\left(u_x+b_2(x)u_{tx}\right)_x+v_t=0,\hspace{0.7cm} &(x,t)\in (-1,1)\times (0,\infty),&\vspace{0.15cm} \\
		v_{tt}-cv_{xx}-u_t=0,\hspace{2.5cm}&(x,t)\in (-1,1)\times (0,\infty),& \vspace{0.15cm}\\
		u(0,t)=v(0,t)=0,u(1,t)=v(1,t)=0,\hspace{0.5cm}& t>0,& \vspace{0.15cm}\\
		u(x,0)=u_0(x),u_{t}(x,0)=u_1(x),\hspace{1cm} &x\in (-1,1),&\vspace{0.15cm}\\ 
		v(x,0)=v_0(x),v_{t}(x,0)=v_1(x),\hspace{1.3cm} &x\in (-1,1),&
	\end{array}\right.
\end{equation}
where $c>0$, and $b_2\in L^{\infty}(-1,1)$ is a non-negative function. They assumed that the damping coefficient is piecewise function in particular they supposed that $b_2(x) = d\mathds{1}_{[0,1]}(x)$, where $d$ is a strictly positive constant. So, they took the damping coefficient to be near the boundary with a global coupling coefficient. They showed the lack of exponential stability and that the semigroup loses speed and it decays polynomially with a rate as $t^{-{\frac{1}{12}}}$. In 2020, Akil, Issa and Wehbe in \cite{2020arXiv200406758A} studied the localized coupled wave equations, by considering the following  system:
\begin{equation*} \left\{	\begin{array}{llll}\vspace{0.15cm}
u_{tt}-\left(au_x +b(x) u_{tx}\right)_x +c(x)y_t =0,& (x,t)\in (0,L)\times (0,\infty) ,&\\ \vspace{0.15cm}
y_{tt}-y_{xx}-c(x)u_t =0,  &(x,t)\in (0,L)\times (0,\infty) ,&\\\vspace{0.15cm}
u(0,t)=u(L,t)=y(0,t)=y(L,t)=0,& t>0 ,& \\\vspace{0.15cm}
(u(x,0),u_t (x,0))=(u_0 (x),u_1 (x)), &x\in (0,L),&\\\vspace{0.15cm}	(y(x,0),y_t (x,0))=(y_0 (x),y_1 (x)), &x\in (0,L)&
\end{array}\right.
\end{equation*}where 
\begin{equation*}
b(x)=\left\{	\begin{array}{lll}\vspace{0.15cm}
b_0 ,& x\in (\alpha_1 ,\alpha_3), &\\
0,&\text{otherwise}&
\end{array}\right. \text{and }\quad
c(x)=\left\{	\begin{array}{lll}\vspace{0.15cm}
c_0 ,& x\in (\alpha_2 ,\alpha_4), &\\
0,&\text{otherwise}&
\end{array}\right.
\end{equation*}\\\linebreak and where $a>0$, $b_0 >0$, $c_0 >0$ and $0<\alpha_1 <\alpha_2 <\alpha_3 <\alpha_4 < L$. They generalized the results of Hassine and Souayeh in \cite{hassine2019stability}  by establishing a polynomial decay rate of type $t^{-1}$. In the same year, Hayek {\it et al.} in \cite{Hayek} studied the stabilization of a multi-dimensional system of  weakly coupled wave equations with one or two locally Kelvin-Voigt damping and non-smooth coefficient at the interface. They established different stability results.
\subsubsection{Wave equations with time delay and without Kelvin-Voigt damping}
 The delay equations of hyperbolic type is given by 
\begin{equation}\label{hyp-delay}
	u_{tt}-\Delta u(x,t-\tau)=0.
\end{equation}
with a delay parameter $\tau>0$.  This system is not well posed since there exists a sequence of solutions tending to infinity for any fixed $t>0$ while the norm of the initial data remain bounded (see Theorem 1.1 in \cite{dreher2009}). In 2006, Nicaise and Pignotti in \cite{Nicaise2006} studied the multidimensional wave equation considering two cases. The first case concerns a wave equation with boundary feedback and a delay term at the boundary 
\begin{equation}\label{EQ-Nicaise-2006-1}
	\left\{
	\begin{array}{ll}
		{u}_{tt}(x,t) -  \Delta u (x,t) = 0, & (x,t)\in  \Omega \times (0,\infty), \nline
		{u}(x,t)=0, & (x,t) \in  \Gamma_D \times (0,\infty), \nline
		\frac{\partial u}{\partial \nu} (x,t) = -\mu_1 u_t(x,t)-\mu_2 u_t(x,t-\tau), & (x,t) \in  \Gamma_N \times (0,\infty), \nline
		\left({u}(x,0),{u}_t(x,0)\right)= \left(u_0(x),u_1(x)\right),   & x\in  \Omega,\nline
		\displaystyle{u_t(x,t) = f_0(x,t) },&\displaystyle{(x,t) \in \Gamma_N \times (-\tau, 0) }.
	\end{array}
	\right.
\end{equation}
The second case concerns a wave equation with an internal feedback and a delayed velocity term ({\it i.e.} an internal delay) and a mixed Dirichlet-Neumann boundary condition 
\begin{equation}\label{EQ-Nicaise-2006-2}
	\left\{
	\begin{array}{ll}
		{u}_{tt}(x,t) -  \Delta u (x,t) + \mu_1 u_t (x,t) + \mu_2 u_t (x, t - \tau) = 0, & (x,t)\in  \Omega \times (0,\infty), \nline
		{u}(x,t)=0, & (x,t) \in  \Gamma_D \times (0,\infty), \nline
		\frac{\partial u}{\partial \nu} (x,t)= 0, & (x,t) \in  \Gamma_N \times (0,\infty), \nline
		\left({u}(x,0),{u}_t(x,0)\right)= \left(u_0(x),u_1(x)\right),   & x\in  \Omega,\nline
		\displaystyle{u_t(x,t) = f_0(x,t) },&\displaystyle{(x,t) \in \Omega \times (-\tau,0) },
	\end{array}
	\right.
\end{equation}
where $\Omega$ is an open bounded domain of $\R^N$  with a boundary $\Gamma$ of class $C^2$ and $\Gamma = \Gamma _ D \cup \Gamma_N$, such that $\Gamma _ D \cap \Gamma_N = \emptyset$. Under the assumption $\mu_2<\mu_1$,  an exponential  decay achieved for the both systems \eqref{EQ-Nicaise-2006-1}-\eqref{EQ-Nicaise-2006-2}. If this assumption does not hold, they found a sequences of delays $\{\tau_k\}_{k}$, $\tau_k\to 0$, for which the corresponding solutions have increasing energy. Furthermore, we refer to \cite{benhassi2009} for the Problem \eqref{EQ-Nicaise-2006-2} in more general abstract setting. In 2010, Ammari {\it et al.} (see \cite{ammari2010}) studied the wave equation with interior delay damping and dissipative undelayed boundary condition in an open domain $\Omega$ of $\R^N,\, N \geq 2.$ The system is described  by: 
\begin{equation}\label{EQ-Ammari-2010}
	\left\{
	\begin{array}{ll}
		{u}_{tt}(x,t) -  \Delta u (x,t) + a u_t (x,t - \tau)=0 , & (x,t)\in  \Omega \times (0,\infty), \nline
		{u}(x,t)=0, & (x,t) \in  \Gamma_0\times (0,\infty), \nline
		\frac{\partial u}{\partial \nu} (x,t)=  - \kappa u_t(x,t), & (x,t) \in  \Gamma_1\times  (0,\infty), \nline
		\left({u}(x,0),{u}_t(x,0)\right)= \left(u_0(x),u_1(x)\right),   & x\in  \Omega,\nline
		\displaystyle{u_t(x,t) = f_0(x,t) },&\displaystyle{(x,t) \in \Omega \times (-\tau,0) }, 
	\end{array}
	\right.
\end{equation}
where $\tau > 0$, $a>0$ and $ \kappa > 0$. Under the condition that $\Gamma_1$ satisfies the $\Gamma$-condition introduced in \cite{Lions88}, they proved that system  \eqref{EQ-Ammari-2010} is uniformly asymptotically stable whenever the delay coefficient is sufficiently small.$\newline$
In 2012, Pignotti in \cite{pignotti2012} considered the wave equation with internal distributed time delay and local damping in a bounded and smooth domain $\Omega \subset \R^N, N \geq 1$.  The considered system is given by the following:
\begin{equation}\label{EQ pignotti 2012}
	\left\{
	\begin{array}{ll}
		{u}_{tt}(x,t) -  \Delta u (x,t) + a \chi_{\omega} u_t(x,t) + \kappa u_t (x , t - \tau) = 0, & (x,t)\in  \Omega \times (0,\infty), \nline
		{u}(x,t)=0, & (x,t) \in  \Gamma \times (0,\infty), \nline
		\left({u}(x,0),{u}_t(x,0)\right)= \left(u_0(x),u_1(x)\right), & x\in  \Omega,\nline
		\displaystyle{u_t(x,t) = f(x,t) },&\displaystyle{(x,t) \in \Omega \times (-\tau,0) }, 
	\end{array}
	\right.
\end{equation}
where $\kappa \in \R$, $\tau>0$, $a>0$ and $\omega$ is the intersection between an open neighberhood of the set $\Gamma_0=\left\{x\in\Omega;\quad (x-x_0)\cdot \nu(x)>0\right\}$ and $\Omega$. Moreover, $\chi_{\omega}$ is the characteristic function of $\omega$. We remark that the damping is localized and it acts on a neighberhood of a part of $\Omega$.  She showed an exponential stability results  if the coefficients of the delay terms satisfy the following assumption $\abs{k}<k_0<a$. \\
Several researches was done on wave equation with time delay acting on the boundary see (\cite{datko1985},\cite{DATKO1990}, \cite{Xu2006}, \cite{Guo2008}, \cite{gugat2010}, \cite{wang2011}, \cite{wang2013}, \cite{Xie2017}) and different type of stability has been proved.
\subsubsection{Wave equations with Kelvin-Voigt damping and time delay}
In 2016, Messaoudi {\it et al.} in \cite{messaoudi2016} considered the stabilization of the following wave equation with strong time delay
\begin{equation*}\label{EQ messaoudi 2016}
\left\{
\begin{array}{ll}
{u}_{tt}(x,t) -  \Delta u (x,t) - \mu_1 \Delta u_t (x,t) - \mu_2 \Delta u_t (x, t - \tau) = 0, & (x,t)\in  \Omega \times (0,\infty), \nline
{u}(x,t)=0, & (x,t) \in  \Gamma \times (0,\infty), \nline
\left({u}(x,0),{u}_t(x,0)\right)= \left(u_0(x),u_1(x)\right),   & x\in  \Omega,\nline
\displaystyle{u_t(x,t) = f_0(x,t) },&\displaystyle{(x,t) \in \Omega \times (-\tau,0) },
\end{array}
\right.
\end{equation*}
where $\mu_1>0$ and $\mu_2$ is a non zero real number.
Under the assumption that $|\mu_2| < \mu_1 $, they obtained an exponential stability result. In 2016, Nicaise {\it et al.} in \cite{Nicaise-Pignotti16} studied the multidimensional wave equation with localized Kelvin-Voigt damping and mixed boundary condition with time delay
\begin{equation}\label{EQ Nicaise 2016}
	\left\{
	\begin{array}{ll}
		{u}_{tt}(x,t) -  \Delta u (x,t) - {\textrm{div}} (a(x) \nabla u_t) = 0, & (x,t)\in  \Omega \times (0,\infty), \nline
		{u}(x,t)=0, & (x,t) \in  \Gamma_0 \times (0,\infty), \nline
		\frac{\partial u}{\partial \nu} (x,t) = -a(x) \frac{\partial u_t}{\partial \nu} (x,t)-\kappa u_t(x,t-\tau), & (x,t) \in  \Gamma_1 \times (0,\infty), \nline
		\left({u}(x,0),{u}_t(x,0)\right)= \left(u_0(x),u_1(x)\right),   & x\in  \Omega,\nline
		\displaystyle{u_t(x,t) = f_0(x,t) },&\displaystyle{(x,t) \in \Gamma_1 \times (-\tau, 0) },
	\end{array}
	\right.
\end{equation}
where $\tau > 0$, $\kappa \in \R$, $a(x) \in L^{\infty}(\Omega)$ and $a(x) \geq a_0 > 0$ on $\omega$ such that $\omega \subset \Omega$ is an open neighborhood of $\Gamma_1$. Under an appropriate geometric condition on $\Gamma_1$ and assuming that $a \in C^{1,1}(\overline{\Omega})$, $\Delta a \in L^{\infty}(\Omega)$, they proved an exponential decay  of the energy of system \eqref{EQ Nicaise 2016}.  In 2019, Anikushyn {\it and al.} in \cite{MR4028788} considered an initial boundary value problem for a viscoelastic wave equation subjected to a strong time localized delay in a Kelvin-Voigt type. The system is given by the following:
\begin{equation*}\label{Eq Anykushyn 2018}
	\left\{
	\begin{array}{ll}
		{u}_{tt}(x,t) -  c_1 \Delta u (x,t) - c_2 \Delta u (x,t - \tau ) - d_1  \Delta u_t (x,t)- d_2 \Delta u_t (x,t - \tau )  = 0, & (x,t)\in  \Omega \times (0,\infty), \nline
		{u}(x,t)=0, & (x,t) \in  \Gamma_0 \times (0,\infty), \nline
		\frac{\partial u}{\partial \nu} (x,t) = 0, & (x,t) \in  \Gamma_1 \times (0,\infty), \nline
		\left({u}(x,0),{u}_t(x,0)\right)= \left(u_0(x),u_1(x)\right),   & x\in  \Omega,\nline
		\displaystyle{u(x,t) = f_0(x,t) },&\displaystyle{(x,t) \in \Omega \times (-\tau, 0) }.
	\end{array}
	\right.
\end{equation*}
Under appropriate conditions on the coefficients, a global exponential decay rate is obtained. In 2015,  Ammari {\it and al.} in \cite{Ammari2015} considered the stabilization problem for an abstract equation with delay and a Kelvin-Voigt damping. The system is given by the following:
\begin{equation*}\label{Abstarct Ammari 2015}
	\left\{
	\begin{array}{ll}
		{u}_{tt}(t) +  a \mathcal{B}\mathcal{B}^* u_t(t) + \mathcal{B}\mathcal{B}^* u(t - \tau),  & t \in   (0,\infty), \nline
		\left({u}(0),{u}_t(0)\right)= \left(u_0,u_1 \right),   \nline
		\displaystyle{\mathcal{B}^* u(t) = f_0(t) },&\displaystyle{t \in  (-\tau,0) }, 
	\end{array}
	\right.
\end{equation*}
for an appropriate class of operator $\mathcal{B}$ and $a > 0.$  Using the frequency domain approach, they obtained an exponential stability result.  \\

Thus, to the best of our knowledge, it seems  to us that there  is no result in the existing literature concerning the case of coupled wave equations with localized Kelvin-Voigt damping and localized time delay, especially in the absence of smoothness of the damping and coupling  coefficients. The goal of the present paper is to fill this gap by studying the stability of system \eqref{sysorig}.\\

The paper is organized as follows: In Section \ref{WPS}, we prove the well-posedness of our system by using semigroup approach. In Section \ref{sec3}, by using a general criteria of Arendt Batty, we show the strong stability of our system in the absence of the compactness of the resolvent. Next, in Section \ref{Pol-Sta-Sec}, by using frequency domain approach  combining with a specific  multiplier method, we prove a polynomial energy decay rate of order $t^{-1}$.

\section{Well-posedness of the System}\label{WPS}
\noindent In this section, we will establish the well-posedness of  system \eqref{sysorig} by using semigroup approach. For this aim, as in \cite{Nicaise2006},
we introduce the following auxiliary change of variable
\begin{equation}
\eta(x,\rho,t):=u_t (x,t-\rho\tau),\quad x\in (0,L),\,\rho\in(0,1),\, t >0.
\end{equation}
Then, system \eqref{sysorig} becomes
\begin{eqnarray}	u_{tt}-\left[au_x +b(x)(\kappa_1 u_{tx}+\kappa_2 \eta_x (x,1,t)) \right]_x  +c(x)y_t =0,& (x,t)\in (0,L)\times (0,\infty) ,&\label{sysorg0}\\ \vspace{0.15cm}
	y_{tt}-y_{xx}-c(x)u_t =0,  &(x,t)\in (0,L)\times (0,\infty) ,&\label{sysorg1}\\\vspace{0.15cm}
	\tau\eta_t (x,\rho,t)+\eta_{\rho}(x,\rho,t)=0, &(x,\rho,t)\in(0,L)\times(0,1)\times (0,\infty) ,&\label{sysorg2}\vspace{0.15cm}
\end{eqnarray}
with the following  boundary conditions 
\begin{equation}\label{diricletbc}
\left\{\begin{array}{lll}
	u(0,t)=u(L,t)=y(0,t)=y(L,t)=0,\quad t>0,\vspace{0.15cm}\\
	\eta(0,\rho,t)=0,\quad (\rho,t)\in (0,1)\times (0,\infty),
	\end{array}
	\right.
	 \end{equation}
	and the following initial conditions
\begin{equation}\label{initialcon}
\left\{\begin{array}{llll}
	u(x,0)=u_0 (x),\qquad u_t (x,0)=u_1 (x), &x\in (0,L),& \vspace{0.15cm}\\
		y(x,0)=y_0 (x),\qquad y_t (x,0)=y_1 (x), &x\in (0,L),&\vspace{0.15cm} \\
		\eta(x,\rho,0) =f_0 (x,-\rho \tau),&(x,\rho)\in (0,L)\times (0,1).& 
\end{array}\right.\end{equation} 
The energy of  system \eqref{sysorg0}-\eqref{initialcon} is given by 
\begin{equation}\label{Energy}
	E(t)=E_1 (t)+E_2 (t)+E_3 (t),
\end{equation}where
\begin{equation*}
E_1 (t)=\frac{1}{2}\int_{0}^{L}\left(\left|u_t \right|^2 +a|u_x |^2 \right)dx,\quad
E_2 (t)=\frac{1}{2}\int_{0}^{L}\left(\left|y_t \right|^2 +|y_x |^2 \right)dx\ \ \text{and}\ \
	E_3 (t)=\frac{ \tau |\kappa_2 |}{2}\intdnb\int_{0}^{1}|\eta_x (\cdot,\rho,t) |^2 d\rho dx. 
	\end{equation*}

\begin{lem}\label{lemenergy}
{\rm	Let  $U=(u,u_t ,y,y_t ,\eta )$ be a regular solution of  system \eqref{sysorg0}-\eqref{initialcon}. Then, the energy $E(t)$  satisfies the following estimation }
	\begin{equation}\label{dewrtt}
	\frac{d}{dt}E(t)\leq - \left(\kappa_1 -\left|\kappa_2\right|\right)\int_{0}^{\beta}|\eta_x (\cdot,0,t)|^2 dx.
	\end{equation}
\end{lem}
\begin{proof}
First, multiplying  \eqref{sysorg0} by $\overline{u_t }$, integrating over $(0,L)$, using integration by parts with  \eqref{diricletbc}, then using the definition of $b(\cdot)$, $c(\cdot)$ and taking the real part, we obtain
\begin{equation}\label{utbar}
  \frac{d}{dt}E_{1} (t)=-\kappa_1  \intdnb|\eta_x (\cdot,0,t)|^2 dx -\Re\left\{\kappa_2 \intdnb \eta_x (\cdot,1,t)\overline{\eta_x }(\cdot,0,t)dx \right\}  - \Re \left\{c_0\intdnc y_t \overline{u_t }dx \right\}.
\end{equation}
Using Young's inequality in \eqref{utbar}, we get 
	\begin{equation}\label{E1'}
	 \frac{d}{dt}E_{1} (t)\leq- \left(\kappa_1 -\frac{|\kappa_2|}{2}\right) \intdnb|\eta_x (\cdot,0,t)|^2 dx+\frac{ |\kappa_2|}{2}   \intdnb |\eta_x (\cdot,1,t)|^2 dx   - \Re \left\{c_0\intdnc y_t \overline{u_t }dx \right\}.
	\end{equation}
	Now,  multiplying  \eqref{sysorg1} by $\overline{y_t }$, integrating over $(0,L)$, using the definition of $c(\cdot)$, then taking the real part, we get
	\begin{equation}\label{E2'}
	 \frac{d}{dt}E_{2} (t)= \Re\left\{c_0 \intdnc u_t \overline{y_t }dx \right\}.
	\end{equation}Deriving \eqref{sysorg2} with respect to $x$, we obtain
	\begin{equation}\label{d/dx b(x)}
	\tau \eta_{xt}(\cdot,\rho,t)+\eta_{x\rho }(\cdot,\rho,t)=0.
	\end{equation}
	Multiplying \eqref{d/dx b(x)} by $|\kappa_2 |\overline{\eta_x }(\cdot,\rho,t)$, integrating over $(0,\beta)\times (0,1)$, then taking the real part, we get	\begin{equation}\label{E3'}
	 \frac{d}{dt}E_{3} (t)=-\frac{ |\kappa_2 |}{2}\intdnb\left(\left|\eta_x (\cdot,1,t)\right|^2-|\eta_x(\cdot,0,t)|^2 \right) dx.
	\end{equation}
	Finally, by adding \eqref{E1'}, \eqref{E2'} and \eqref{E3'}, we obtain \eqref{dewrtt}. The proof is thus complete.
\end{proof}\\\linebreak
In the sequel, the assumption on $\kappa_1$ and $\kappa_2$ will ensure that 
\begin{equation}\tag{${\rm H}$}\label{H}
\kappa_1>0,\quad \kappa_2\in\R^{\ast}\quad \text{and} \quad \abs{\kappa_2}<\kappa_1.
\end{equation}	
Under the hypothesis \eqref{H} and  from Lemma \ref{lemenergy},  the system \eqref{sysorg0}-\eqref{initialcon} is dissipative in the sense that its energy is non-increasing with respect to time (i.e. $E^{\prime}(t)\leq 0$). Let us define the Hilbert  space $\HH$ by
\begin{equation*}
	\HH:=\left( H^{1}_{0}(0,L)\times L^2 (0,L)\right)^2 \times \mathcal{W}, 
\end{equation*}where \begin{equation*}
 \mathcal{W}:=	L^2 ((0,1);H^1_L (0 ,\beta )) \ \ \text{and}\ \  H^1_L (0,\beta):=\left\{\widetilde{\eta}\in H^1 (0 ,\beta )\ |\ \widetilde{\eta}(0)=0 \right\}. 
 \end{equation*}
 The space 
$\mathcal{W}$ is an Hilbert space of $H^1_L (0 ,\beta )$-valued functions on $(0,1)$, equipped with the following inner product
\begin{equation*}
(\eta^1 ,\eta^2  )_{\mathcal{W}}:=\intdnb\int_{0}^{1}\eta^1_x  \overline{\eta^2_x }d\rho dx, \quad \forall \, \eta^1, \eta^2 \in \mathcal{W}.
\end{equation*}
The Hilbert space $\mathcal{H}$ is equipped with the following inner product 
\begin{equation}\label{p1-norm}
\left(U,U^1 \right)_{\HH}=\intdx \left(au_x \overline{u_{x}^1 }+v\overline{v^1 }+y_x \overline{y_{x}^1 }+z\overline{z^1 }\right)dx + \tau |\kappa_2 |\intdnb\int_{0}^{1}\eta_x (\cdot,\rho) \overline{\eta_{x}^1 } (\cdot,\rho) d\rho dx,
\end{equation}
where  $U=(u,v,y,z,\eta (\cdot,\rho))^{\top} $, $U^1 =(u^1 ,v^1 ,y^1 ,z^1 ,\eta^1 (\cdot,\rho))^{\top}\in\HH$. Now, we define the linear unbounded  operator $\AA:D(\AA)\subset \HH\longmapsto \HH$  by:
\begin{equation}
D(\AA)=\left\{\begin{array}{cc}\vspace{0.25cm}
U=(u,v,y,z,\eta (\cdot,\rho))^{\top}\in \HH \,\,|\,\, y\in H^2 (0,L)\cap H^{1}_0 (0,L),\,\, v,z \in H^{1}_0 (0,L)\\\vspace{0.25cm}
 (S_b (u,v,\eta))_x  \in L^2 (0,L),\quad \eta_{\rho} (\cdot,\rho)\in \mathcal{W},\quad  \eta(\cdot,0)=v(\cdot)
\end{array}\right\} 
\end{equation}and 
\begin{equation}\label{p1-opA}
	\AA\begin{pmatrix}
	u\\v\\y\\z\\\eta (\cdot,\rho)
	\end{pmatrix}=
	\begin{pmatrix} 
	v\\ (S_b (u,v,\eta))_x  -c(\cdot)z\\z\\y_{xx} +c(\cdot)v\\-\tau^{-1}\eta_{\rho}(\cdot,\rho)
	\end{pmatrix},
\end{equation}where $S_b (u,v,\eta):= au_x +b(\cdot)\left(\kappa_1 v_x +\kappa_2 \eta_x (\cdot,1) \right)$. Moreover, from the definition of $b(\cdot)$, we have \begin{equation}S_b (u,v,\eta) =\left\{	\begin{array}{lll}\vspace{0.15cm}
S_1 (u,v,\eta),& x\in (0 ,\beta), &\\
au_x ,&x\in (\beta ,L),&
\end{array}\right.
\end{equation}
where $S_1 (u,v,\eta) :=au_x +\kappa_{1} v_x +\kappa_2 \eta_x (\cdot,1).$ Now, if $U=(u, u_t ,y, y_t ,\eta(\cdot,\rho))^{\top}$, then system \eqref{sysorg0}-\eqref{initialcon} can be written as the following first order evolution equation 
\begin{equation}\label{firstevo}
	U_t =\AA U , \quad U(0)=U_0,
\end{equation}where  $U_0 =(u_0 ,u_1 ,y_0 ,y_1 ,f_0 (\cdot,-\rho\tau) )^{\top}\in \HH$.
\begin{rk}\label{Ainj}
	{\rm
The linear unbounded operator $\AA$ is injective (i.e. 	$\ker(\AA)=\{0\}$). Indeed, if $U\in D(\AA)$ such that $\AA U=0$, then $v=z=\eta_{\rho}(\cdot,\rho)=0$ and since $\eta (\cdot,0)=v(\cdot)$, we get $\eta(\cdot,\rho)=0$. Consequently, $\left(S_b (u,v,\eta)\right)_x =u_{xx}=0$ and $y_{xx}=0$. Finally, since $u(0)=u(L)=y(0)=y(L)=0$, then $u=y=0$. Thus,  $U=(u,v,y,z,\eta(\cdot,\rho))^{\top}=0$. \xqed{$\square$}

}
\end{rk}
\begin{prop}\label{amdissip}{\rm
Under the hypothesis \eqref{H}, the unbounded linear operator $\AA$ is m-dissipative in the energy space $\HH$.} 
\end{prop}
\begin{proof}
For all $U=(u,v,y,z,\eta(\cdot,\rho))^{\top}\in D(\AA)$, from \eqref{p1-norm} and \eqref{p1-opA}, we have
	$$
\begin{array}{lll}
\displaystyle	\Re(\AA U,U)_{\HH}=\displaystyle\Re\left\{\intdx av_x \overline{u_x}dx\right\}+\Re\left\{\intdx\left(S_b (u,v,\eta)\right)_x \overline{v}dx\right\}+\Re\left\{\intdx z_x \overline{y_x}dx\right\}+\Re\left\{\intdx y_{xx} \overline{z}dx\right\}\vspace{0.25cm}\\ \hspace{2cm}\displaystyle -\, \Re \left\{|\kappa_{2}|\int_{0}^{\beta}\int_{0}^{1}\eta_{x\rho}(\cdot,\rho)\overline{\eta_x}(\cdot,\rho)d\rho dx\right\}.
\end{array}
$$
Using integration by parts to the second and fourth terms in the above equation, then using the definition of $S_b(u,v,\eta)$ and the fact that $U\in D(\AA)$, we get
\begin{equation*}
	\Re (\AA U,U)_{\HH}=-\kappa_1 \intdnb |v_x|^2 dx -\Re\left\{\kappa_2 \intdnb \eta_x (\cdot,1)\overline{v_x}dx \right\}-\frac{|\kappa_2|}{2}\intdnb \int_{0}^{1}\frac{d}{d\rho}|\eta_x (\cdot,\rho)|^2 d\rho dx,
\end{equation*}
the fact that $\eta(\cdot,0)=v(\cdot)$, implies that
	\begin{equation*}\label{reauu}
		\Re \left(\AA U,U\right)_{\HH} =- \left(\kappa_1 -\frac{|\kappa_2|}{2}\right)\intdnb |v_x |^2 dx-\frac{ |\kappa_2 |}{2}\intdnb |\eta_x (\cdot,1)|^2 dx- \Re\left\{\kappa_2\intdnb \eta_x (\cdot,1)\overline{v_x }dx\right\}.
		\end{equation*}
 Using Young's inequality in the above equation and the hypothesis \eqref{H}, we obtain  \begin{equation}\label{dissiparive}
	\Re \left(\AA U,U\right)_{\HH} \leq- \left(\kappa_1 -|\kappa_2|\right)\intdnb |v_x |^2 dx\leq 0,
	\end{equation}
	which implies that $\AA$ is dissipative. Now, let us prove that  $\AA$ is maximal. For this aim, let $F=(f^1 ,f^2 ,f^3 ,f^4 ,f^5 (\cdot,\rho) )^{\top}\in \HH$, we look for $U=(u,v,y,z,\eta (\cdot,\rho))^{\top}\in D(\AA)$ unique solution  of \begin{equation}\label{-au=f}
		-\AA U=F.
	\end{equation}
	Equivalently, we have the following system
	\begin{eqnarray}
		-v&=&f^1 , \label{f1}\\
-(S_b (u,v,\eta))_x +c(\cdot)z &=&f^2 ,\label{f2}\\
		-z&=&f^3 , \label{f3}\\
		-y_{xx}-c(\cdot)v&=&f^4 , \label{f4}\\
		\tau^{-1}\eta_{\rho}(\cdot,\rho)&=&f^5  (\cdot,\rho),\label{f5}
	\end{eqnarray}
	with the following boundary conditions \begin{equation}\label{boundary1mdissip}
		u(0)=u(L)=y(0)=y(L)=0,\quad \eta(0,\rho)=0 \quad \text{and} \quad \eta (\cdot,0)=v(\cdot). 
	\end{equation}
From \eqref{f1}, \eqref{f5} and \eqref{boundary1mdissip}, we get
	\begin{equation}\label{eta}
		 \eta(x,\rho)=\tau\int_{0}^{\rho}f^5 (x,s)ds-f^1 , \quad (x,\rho)\in (0,L )\times (0,1).
	\end{equation}
	Since, $f^1 \in H^{1}_0 (0,L) $ and $f^5 (\cdot,\rho)\in \mathcal{W}$. Then, from \eqref{f5} and \eqref{eta}, we get $\eta_{\rho} (\cdot,\rho), \eta(\cdot,\rho) \in \mathcal{W}$.
	Now, see the definition of $S_b(u,v,\eta)$, substituting \eqref{f1}, \eqref{f3} and \eqref{eta} in \eqref{f2} and \eqref{f4}, we get the following system \begin{eqnarray}
		\left[au_x +b(\cdot)\left(-\kappa_1 f^{1}_x +\tau \kappa_2  \int_{0}^{1}f^{5}_x (\cdot,s)ds-\kappa_{2}f^{1}_x \right)\right]_x +c(\cdot)f^3 &=&-\,f^2, \label{ulax}\\ y_{xx}-c(\cdot)f^1&=&-\,f^4,\label{ylax}\\\vspace{0.25cm}
		u(0)=u(L)=y(0)=y(L)&=&0. \label{bculaxylax}
	\end{eqnarray}
Let $(\phi ,\psi) \in H^{1}_0 (0,L)  \times H^{1}_0 (0,L)$. Multiplying \eqref{ulax} and \eqref{ylax} by $\overline{\phi}$ and $\overline{\psi}$ respectively, integrating over $(0,L)$, then using integrations by parts, we obtain
\begin{equation}\label{ulax1}
	\displaystyle	a\intdx u_x \overline{\phi_x }dx=	\displaystyle \intdx f^2 \overline{\phi}dx+c_0 \intdnc f^3 \overline{\phi}dx+ (\kappa_1 +\kappa_2 )\intdnb f^{1}_x \overline{\phi_x }dx-	\displaystyle  \tau \kappa_2 \intdnb\left( \int_{0}^{1} f^{5}_x (\cdot,s)ds \right) \overline{\phi_x }dx
	\end{equation}
and 
	\begin{equation}\label{ylax2}
		\intdx y_x \overline{\psi_x }dx=\intdx f^4 \overline{\psi}dx-c_0 \intdnc f^1 \overline{\psi}dx.
	\end{equation} 
Adding  \eqref{ulax1} and \eqref{ylax2}, we obtain
	\begin{equation}\label{vf}
		\mathcal{B}((u,y),(\phi,\psi))=\mathcal{L}(\phi,\psi), \quad\forall (\phi,\psi)\in H^{1}_0 (0,L)\times  H^{1}_0 (0,L), 
	\end{equation} 
	where $$\mathcal{B}((u,y),(\phi,\psi))=a\displaystyle\intdx u_x \overline{\phi_x }dx+\intdx y_x \overline{\psi_x }dx$$ and $$ \displaystyle \mathcal{L}(\phi,\psi)=	\displaystyle \intdx \left(f^2 \overline{\phi}+f^4 \overline{\psi}\right)dx+c_0 \intdnc \left(f^3 \overline{\phi}-f^1 \overline{\psi}\right)dx- \tau \kappa_2 \intdnb\left( \int_{0}^{1} f^{5}_x (\cdot,s)ds \right) \overline{\phi_x }dx + 	\displaystyle(\kappa_1 +\kappa_2 )\intdnb f^{1}_x \overline{\phi_x }dx.  $$
	It is easy to see that,  $\mathcal{B}$ is a sesquilinear, continuous and coercive form on $\left( H^{1}_0 (0,L) \times H^{1}_0 (0,L)\right)^2 $ and $\mathcal{L}$ is a linear and continuous form on $ H^{1}_0 (0,L) \times H^{1}_0 (0,L)$. Then, it follows by Lax-Milgram theorem that \eqref{vf} admits a unique solution $(u,y)\in H^{1}_0 (0,L)\times H^{1}_0 (0,L) $. By using the classical elliptic regularity, we deduce that  system \eqref{ulax}-\eqref{bculaxylax} admits a unique solution  $(u,y)\in H^{1}_0 (0,L)\times \left(H^2 (0,L)\cap H^{1}_0 (0,L)\right) $ such that $(S_b (u,v,\eta))_x \in L^2(0,L)$ and since $\ker(\AA)=\{0\}$ (see Remark \ref{Ainj}), we get  $\displaystyle U=\left(u,-f^1,y,-f^3,\tau\int_{0}^{\rho}f^5 (\cdot,s)ds-f^1\right)^{\top} \in D(\AA) $ is a unique solution of \eqref{-au=f}. Then, $\mathcal{A}$ is an isomorphism and since $\rho\left(\mathcal{A}\right)$ is open set of $\mathbb{C}$ (see Theorem 6.7 (Chapter III) in \cite{Kato01}),  we easily get $R(\lambda I -\mathcal{A} ) = {\mathcal{H}}$ for a sufficiently small $\lambda>0 $. This, together with the dissipativeness of $\mathcal{A}$, imply that   $D\left(\mathcal{A}\right)$ is dense in ${\mathcal{H}}$   and that $\mathcal{A}$ is m-dissipative in ${\mathcal{H}}$ (see Theorems 4.5, 4.6 in  \cite{Pazy01}). The proof is thus complete.
\end{proof}\\\linebreak
According to Lumer-Philips theorem (see \cite{Pazy01}), Proposition \ref{amdissip} implies that the operator $\AA$ generates a $C_{0}$-semigroup of contractions $e^{t\AA }$ in $\HH$ which gives the well-posedness of \eqref{firstevo}. Then, we have the following result:
\begin{Thm}{\rm
	Under hypothesis \eqref{H}, for all $U_0 \in \HH$,  System \eqref{firstevo} admits a unique weak solution $$U(x,\rho,t)=e^{t\AA}U_0(x,\rho) \in C^0 (\R^+ ,\HH).
	$$ Moreover, if $U_0 \in D(\AA)$, then the system \eqref{firstevo} admits a unique strong solution $$U(x,\rho,t)=e^{t\AA}U_0(x,\rho) \in C^0 (\R^+ ,D(\AA))\cap C^1 (\R^+ ,\HH).$$}
\end{Thm}
\section{Strong Stability}\label{sec3}\noindent In this section, we will prove the strong stability of  system \eqref{sysorg0}-\eqref{initialcon}. The main result of this section is the following theorem.
\begin{theoreme}\label{p1-strongthm2}
	{\rm	Assume that \eqref{H} is true. Then, the $C_0-$semigroup of contraction $\left(e^{t\AA}\right)_{t\geq 0}$ is strongly stable in $\HH$; i.e., for all $U_{0}\in \HH$, the solution of \eqref{firstevo} satisfies 
		$$
		\lim_{t\rightarrow +\infty}\|e^{t\AA}U_{0}\|_{\HH}=0.
		$$}
\end{theoreme}\noindent According to Theorem \ref{App-Theorem-A.2}, to prove Theorem \ref{p1-strongthm2}, we need to prove that the operator $\AA$ has no pure imaginary eigenvalues and $\sigma(\AA)\cap i\R $ is countable. The proof of Theorem \ref{p1-strongthm2} will be achieved from the following proposition.
\begin{prop}\label{pol-prop1}{\rm
		Under the hypothesis \eqref{H}, we have 
		\begin{equation}\label{Condition-1-pol}
		i\R\subset \rho(\mathcal{A}).
		\end{equation}}
\end{prop}
\noindent  We will prove Proposition  \ref{pol-prop1} by contradiction argument. Remark that, it has been proved in Proposition \ref{amdissip} that $0\in \rho(\mathcal{A})$. Now, suppose that \eqref{Condition-1-pol} is false, then there exists $\omega\in \R^{\ast}$ such that $i\omega\notin \rho(\mathcal{A})$.  According to Remark \ref{App-Lemma-A.3}, let $\left\{(\la^n,U^n:=(u^n,v^n,y^n,z^n,\eta^n(\cdot,\rho))^{\top})\right\}_{n\geq 1}\subset \R^{\ast}\times D(\mathcal{A})$, with
\begin{equation}\label{contra1}
\la^n\to \omega\ \text{as}\ n\to \infty\ \ \text{and}\ \  \abs{\la^n}<\abs{\omega}\end{equation}
and 
\begin{equation}\label{contra2}
\|U^n\|_{\mathcal{H}}=\left\|(u^n ,v^n ,y^n ,z^n ,\eta^n (\cdot, \rho))^{\top}\right\|_{\HH}=1,
\end{equation}
such that
\begin{equation}\label{eq0}
(i\la^n I -\AA ) U^n =F^n:=(f^{1,n},f^{2,n},f^{3,n},f^{4,n},f^{5,n}(\cdot,\rho))^{\top}  \to 0  \quad \text{in}\quad \HH.
\end{equation}
Equivalently, we have 
\begin{eqnarray}
i\la^n u^n -v^n &=& f^{1,n} \to 0 \quad\qquad \text{in}\quad H^{1}_0 (0,L),\label{eq1ss}\\
i\la^n v^n -\left(S_b (u^n,v^n,\eta^n) \right)_x   +c(\cdot)z^n &=&f^{2,n} \to 0 \quad\qquad \text{in}\quad  L^2 (0,L),\label{eq2ss}\\
i\la^n y^n -z^n &=& f^{3,n} \to 0 \quad\qquad \text{in}\quad H^{1}_0 (0,L),\label{eq3ss}\\
i\la^n z^n -y_{xx}^n -c(\cdot)v^n&=&f^{4,n}\to 0 \quad\qquad \text{in}\quad L^2  (0,L),\label{eq4ss}\\i\la^n \eta^n (.,\rho) +\tau^{-1}\eta_{\rho}^n (\cdot,\rho) &=&f^{5,n}(\cdot,\rho)\to 0 \quad \text{in}\quad \mathcal{W}.\label{eq5ss}
\end{eqnarray}
Then, we will proof condition \eqref{Condition-1-pol} by finding a contradiction with \eqref{contra2} such as $\|U^n\|_{\mathcal{H}}\to 0$. The proof of proposition \ref{pol-prop1} has been divided into several Lemmas.
\begin{lem}\label{firstlemmass} {\rm Under the  hypothesis \eqref{H}, the solution $U^n =(u^n ,v^n ,y^n ,z^n ,\eta^n (\cdot,\rho))^{\top}\in D(\AA )$ of system \eqref{eq1ss}-\eqref{eq5ss} satisfies the following limits}
	\begin{eqnarray}
	\lim_{n \to \infty} \intdnb |v^{n}_x |^2 dx =0, \label{limvnx}\\
	\lim_{n \to \infty} \intdnb |v^n |^2 dx =0 ,\label{limvn}\\
	\lim_{n \to \infty} \intdnb |u^{n}_x |^2 dx =0, \label{limunx}\\
	\lim_{n \to \infty} \intdnb \int_{0}^{1} |\eta^{n}_x (\cdot,\rho) |^2 d\rho dx =0, \label{limetan}\\
	\lim_{n \to \infty} \intdnb  |\eta^{n}_x (\cdot,1)|^2  dx =0,\label{limeta1}\\
	\lim_{n \to \infty} \intdnb |S_1(u^n,v^n,\eta^n) |^2 dx =0.\label{limbig}
	\end{eqnarray}
\end{lem}
\begin{proof}
	First, taking the inner product of \eqref{eq0} with $U^n$ in $\HH$ and  using \eqref{dissiparive} with the help of hypothesis \eqref{H}, we obtain
	\begin{equation}\label{3.13}\intdnb |v^{n}_x |^2 dx \leq -\frac{1}{\kappa_1 -|\kappa_2 |}\Re (\AA U^n  ,U^n )_{\HH}=\frac{1}{\kappa_1 -|\kappa_2 |}\Re (F^n  ,U^n )_{\HH} \leq \frac{1}{\kappa_1-|\kappa_2 |}\|F^n \|_{\HH} \|U^n \|_{\HH}.  \end{equation} Then, by passing to the limit in \eqref{3.13} and by using the fact that  $\|U^n \|_{\HH}=1 $ and $\|F^n \|_{\HH} \to 0 $, we obtain \eqref{limvnx}.
	Now, since $v^n \in H^1_0 (0,L)$, then it follows  from Poincar\'e inequality that there exists a constant $C_p >0 $ such that 
	\begin{equation}\label{gpoincare}
	\|v^n \|_{L^2 (0 ,\beta )}\leq C_p \|v^{n}_x \|_{L^2 (0 ,\beta)}.
	\end{equation}
	Thus, From \eqref{limvnx} and \eqref{gpoincare}, we obtain \eqref{limvn}. Next, 
	from   \eqref{eq1ss} and  the fact that  $\displaystyle\intdnb |f^{1,n}_x |^2 dx \leq \intdx |f^{1,n}_x |^2 dx  \leq a^{-1} \|F^n \|_{\HH}^2 $, we deduce that 
	\begin{equation}\label{3.16}	
	\displaystyle \intdnb |u^{n}_x |^2 dx \leq\displaystyle \frac{2}{(\la^n )^2 } \intdnb |v^{n}_x |^2 dx +\frac{2}{(\la^n )^2 }\intdnb |f^{1,n}_x |^2 dx
	\leq  \displaystyle \frac{2}{(\la^n )^2 } \intdnb |v^{n}_x |^2 dx +\frac{2}{a(\la^n )^2 } \|F^n \|_{\HH}^2 .
	\end{equation}
	Therefore, by passing to the limit in  \eqref{3.16} and by using \eqref{contra1}, \eqref{limvnx} and the fact that $\|F^n \|_{\HH}\to0$, we obtain \eqref{limunx}.
	Moreover, from  \eqref{eq5ss} and the fact that  $\eta^n (\cdot,0)=v^n (\cdot)$, we deduce that 
	\begin{equation}\label{etan}
	\eta^n (x,\rho)=v^n e^{-i\la^n \tau \rho}+\tau\int_{0}^{\rho}e^{i\la^n \tau (s-\rho)}f^{5,n}(x,s)ds,\quad \quad (x,\rho)\in(0,L )\times (0,1).
	\end{equation}
	From  \eqref{etan} and the fact that $\displaystyle\intdnb \int_{0}^{1} |f^{5,n}_x (\cdot,s)|^2 ds dx \leq \tau^{-1} |\kappa_{2}|^{-1}\|F^n \|_{\HH}^2 $, we obtain \begin{equation}\label{3.18}\begin{array}{llll}
	\displaystyle\intdnb \int_{0}^{1} |\eta^{n}_x (\cdot,\rho)|^2 d\rho dx  \leq \displaystyle 2\intdnb |v^{n}_x   |^2 dx +2\tau^2 \intdnb \int_{0}^{\rho} \int_{0}^{1}\rho |f^{5,n}_x (\cdot,s)|^2 d\rho ds dx\vspace{0.25cm}\\
\hspace{3.5cm}	\leq\displaystyle 2\intdnb |v^{n}_x   |^2 dx +\tau^2 \intdnb \int_{0}^{1}  |f^{5,n}_x (\cdot,s)|^2  ds dx \vspace{0.25cm}\\\hspace{3.5cm}
	\leq \displaystyle 2\intdnb |v^{n}_x   |^2 dx +\tau|\kappa_2|^{-1}\|F^n \|_{\HH}^2 .
	\end{array}\end{equation}Thus, by passing to the limit in \eqref{3.18} and by using  \eqref{limvnx} with the fact that $\|F^n \|_{\HH}\to 0$, we obtain \eqref{limetan}.
	On the other hand,	from  \eqref{etan}, we have \begin{equation*}
	\eta^{n}_{x} (\cdot,1)=v^{n}_{x}e^{-i\la^n \tau}+\tau \int_{0}^{1}e^{i\la^n \tau (s-1)}f^{5,n}_{x}(\cdot,s)ds, 
	\end{equation*}
	consequently,  by using the same argument as proof of \eqref{limetan}, we obtain \eqref{limeta1}. 
	Next, it is clear to see that 
 \begin{equation*}\label{itsclear}
		\intdnb |S_1 (u^n,v^n,\eta^n) |^2 dx =\intdnb |au^{n}_x +\kappa_{1}v^{n}_x +\kappa_{2} \eta^{n}_x (\cdot ,1 )|^2 dx  \leq 3a^2 \intdnb|u^n_{x}|^2 dx +3\kappa_{1}^2 \intdnb |v^{n}_{x}|^2 dx +3\kappa_{2}^2 \intdnb |\eta^{n}_{x}(\cdot,1)|^2 dx.
		\end{equation*}
	Finally, passing to the limit in the above estimation, then using  \eqref{limvnx}, \eqref{limunx} and \eqref{limeta1}, we obtain \eqref{limbig}. The proof is thus complete.
	
\end{proof}\\\linebreak
Now we fix a  function $g \in C^1 \left([\alpha ,\beta ]\right)$ such that
\begin{equation}
g(\alpha)=-g(\beta)=1\quad \text{and \ set} \quad \max_{x\in [\alpha ,\beta ]}|g(x)|=M_{g} \,\,\, \text{and}\,\, \max_{x\in [\alpha ,\beta ]}|g^{\prime}(x)|=M_{g^{\prime}}. 
\end{equation}
\begin{rk}{\rm
		To prove the existence of a function $g$, we need to find an example. For this aim,  we can take \\ $\displaystyle g(x)=1+\frac{2(\alpha -x )}{\beta -\alpha }$, then  $g\in C^1 ([\alpha ,\beta])$,  $\displaystyle g(\alpha )=-g(\beta )=1$,  $\displaystyle M_g =1$ and  $\displaystyle M_{g^{\prime}}=\frac{2}{\beta -\alpha }$.
		Also, we can take $\displaystyle g(x)=\cos\left(\frac{(\alpha -x)\pi}{\alpha -\beta } \right).$}
	\xqed{$\square$}
\end{rk}
\begin{lem}{\rm Under the  hypothesis \eqref{H}, the solution  $U^n =(u^n ,v^n ,y^n ,z^n ,\eta^n (\cdot,\rho) )^{\top}\in D(\AA )$ of  system \eqref{eq1ss}-\eqref{eq5ss} satisfies the following inequalities
		\begin{eqnarray}
		|z^n (\beta)|^2 +|z^n (\alpha)|^2 \leq M_{g^{\prime}} \int_{\alpha }^{\beta }|z^n |^2 dx +2|\la^n | M_{g} \left(\int_{\alpha}^{\beta}|z^n |^2 dx\right)^{\frac{1}{2}}+2M_{g}\|F^n \|_{\HH}\label{zboundary},\\
		|y^{n}_{x}(\beta)|^2 +|y^{n}_{x}(\alpha)|^2 \leq M_{g^{\prime}}\int_{\alpha}^{\beta} |y^{n}_{x}|^2 dx +2(|\la^n | +c_0 )M_{g}\left(\int_{\alpha}^{\beta} |y^{n}_{x}|^2 dx \right)^{\frac{1}{2}}+2M_{g}\|F^n \|_{\HH}\label{ynxboundary}
		\end{eqnarray}and the following limits
		\begin{equation}\label{limvnboundary}
		\lim_{n\to \infty} \left(\left|v^n (\beta )\right|^2 +\left|v^n (\alpha )\right|^2  \right) =0,
		\end{equation}
		\begin{equation}\label{limbigbig}
		\lim_{n\to \infty}\left(\left|\left(S_{1} (u^n,v^n,\eta^n)\right) (\beta^- )\right|^2 + \left|\left(S_{1} (u^n,v^n,\eta^n)\right) (\alpha )\right|^2 \right ) =0.
		\end{equation} }
	
	
\end{lem}
\begin{proof} First, 
	from  \eqref{eq3ss}, we deduce that
	\begin{equation}\label{3x}
	i\la^n y^{n}_x -z^{n}_x =f^{3,n}_x .
	\end{equation} 
	Multiplying  \eqref{3x} and \eqref{eq4ss} by $2g\overline{z^n }$ and $2g\overline{y^{n}_x }$ respectively, integrating over $(\alpha ,\beta)$, using the definition of $c(\cdot)$, then taking the real part, we get
	\begin{equation}\label{3xm}
	\Re \left\{ 2i\la^n \intdni gy^{n}_{x}\overline{z^n }dx\right\}-\int_{\alpha }^{\beta }g\left(\left|z^n \right|^2 \right)_x dx =\Re\left\{ 2\intdni gf^{3,n}_x \overline{z^n }dx\right\}
	\end{equation}and
	\begin{equation}\label{3.27}
	\Re\left\{2i\la^n \intdni gz^n \overline{y^{n}_x }dx \right\}-\intdni g\left(\left|y^{n}_x \right|^2 \right)_x dx-\Re\left\{ 2c_0 \intdni gv^n \overline{y^{n}_x }dx\right\}=\Re \left\{2\intdni g f^{4,n}\overline{y^{n}_x }dx\right\}.
	\end{equation}
	Using integration by parts in  \eqref{3xm} and \eqref{3.27}, we obtain
	\begin{equation*}\label{3xmm}
	\left[-g\left|z^n \right|^2 \right]_{\alpha }^{\beta } =-	\intdni g^{\prime}|z^n |^2 dx-\Re \left\{ 2i\la^n \intdni gy^{n}_{x}\overline{z^n }dx\right\}+\Re\left\{ 2\intdni gf^{3,n}_x \overline{z^n }dx\right\}
	\end{equation*}
	and
		\begin{equation*}\label{3.29}
		\left[-g\left|y^{n}_x \right|^2 \right]_{\alpha }^{\beta }=-\intdni g^{\prime}|y^{n}_x |^2 dx -\Re\left\{2i\la^n \intdni gz^n \overline{y^{n}_x }dx  \right\}+\Re\left\{2c_0 \intdni gv^n \overline{y^{n}_x } \right\}+\Re\left\{2\intdni gf^{4,n} \overline{y^{n}_x }dx  \right\}.
		\end{equation*}
	Using the definition of $g$ and Cauchy-Schwarz inequality in  the above equations, we obtain
	\begin{equation}\label{3.28}	\begin{array}{lll}
	\displaystyle	|z^n (\beta )|^2 +|z^n (\alpha )|^2 \leq\displaystyle M_{g^{\prime}}\intdni |z^n |^2 dx +2|\la^n | M_{g}\left( \intdni |y^{n}_x |^2 dx\right)^{\frac{1}{2}}\left( \intdni |z^{n}  |^2 dx\right)^{\frac{1}{2}}\vspace{0.25cm}\\\hspace{3.25cm}
	+\, \displaystyle2 M_{g}\left( \intdni |f^{3,n}_x  |^2 dx\right)^{\frac{1}{2}}\left( \intdni |z^{n}  |^2 dx\right)^{\frac{1}{2}}
	\end{array}\end{equation}and
	\begin{equation}\label{3.31}
	\begin{array}{lll}
	\displaystyle	|y^{n}_x (\beta )|^2 +|y^{n}_x (\alpha )|^2 \leq \displaystyle M_{g^{\prime}} \intdni |y^{n}_x |^2 dx +2|\la^n |M_g \left(\intdni|y^{n}_x |^2 dx\right)^{\frac{1}{2}} \left(\intdni|z^{n} |^2 dx\right)^{\frac{1}{2}}\vspace{0.25cm}\\
\hspace{3.25cm}	+\, \displaystyle2c_0 M_g \left(\intdni|y^{n}_x |^2 dx\right)^{\frac{1}{2}} \left(\intdni|v^{n} |^2 dx\right)^{\frac{1}{2}} \vspace{0.25cm}\\\hspace{3.25cm} +\, \displaystyle2M_g \left(\intdni|f^{4,n} |^2 dx\right)^{\frac{1}{2}} \left(\intdni|y^{n}_x  |^2 dx\right)^{\frac{1}{2}} .
	\end{array}
	\end{equation}
	Therefore, from  \eqref{3.28},  \eqref{3.31} and the fact that $\displaystyle \intdni |\xi_{1}^n |^2 dx \leq \intdx |\xi_{1}^n |^2 dx \leq \|U^n \|_{\HH}^2 =1$ with $\xi_{1}^n \in \{v^n ,y^{n}_x ,z^n \}$  and $\displaystyle \intdni |\xi_{2}^n  |^2 dx \leq \intdx |\xi_{2}^n|^2 dx \leq \|F^n \|_{\HH}^2 $ with $\xi_{2}^n \in \{f^{3,n}_x ,f^{4,n} \}$, we obtain  \eqref{zboundary} and \eqref{ynxboundary}. 
	On the other hand, from  \eqref{eq1ss}, we deduce that 
	\begin{equation}\label{1x1}
	i\la^n u^{n}_x -v^{n}_x =f^{1,n}_x .
	\end{equation}
	Multiplying  \eqref{1x1} and \eqref{eq2ss} by $2g\overline{v^n }$ and $2g\overline{S_1 }(u^n,v^n,\eta^n)$ respectively, integrating over $(\alpha ,\beta)$, using the definition of $c(\cdot)$ and $S_b (u^n,v^n,\eta^n) $, then taking the real part, we get 
	\begin{equation}\label{3.32}
	\Re\left\{2i\la^n \intdni gu^n_x \overline{v^n}dx \right\}-\intdni g(|v^n |^2 )_x dx =\Re\left\{2\intdni g f^{1,n}_x \overline{v^n}dx \right\}\end{equation}and\begin{equation}\label{3.33}
	\begin{array}{lll}
\displaystyle 	\Re \left\{ 2i\la^n  \intdni gv^n \overline{S_{1} }(u^n,v^n,\eta^n)dx\right\}	-\intdni g\left(\left|S_{1}(u^n,v^n,\eta^n)\right|^2 \right)_x   dx\vspace{0.25cm}\\
\displaystyle
+\,\Re\left\{2c_0  \intdni g z^n \overline{S_{1}  }(u^n,v^n,\eta^n)dx \right\}	=\Re \left\{2\intdni gf^{2,n}\overline{S_{1} }(u^n,v^n,\eta^n) dx\right\}.
	\end{array}
	\end{equation} Using integration by parts in  \eqref{3.32} and \eqref{3.33}, we get 
	\begin{equation*}\label{3.35k}
	\left[-g\left|v^n \right|^2  \right]_{\alpha }^{\beta } =-\intdni g^{\prime}|v^n |^2 dx -\Re \left\{2i\la^n \intdni gu^{n}_x \overline{v^n }dx\right\} +\Re \left\{2\intdni gf^{1,n}_x \overline{v^n }dx \right\}
	\end{equation*}
	and
\begin{equation*}\label{3.36k}
\begin{array}{lll}
	\displaystyle 	\left[ -g \left|S_1 (u^n,v^n,\eta^n)\right|^2 \right]_{\alpha }^{\beta }=-\intdni g^{\prime}\left|S_1 (u^n,v^n,\eta^n) \right|^2 dx -\Re\left\{2i\la^n \intdni g v^n \overline{S_1 }(u^n,v^n,\eta^n)dx \right\}\vspace{0.25cm}\\
	\displaystyle
	\hspace{3.5cm}	-\,\Re\left\{2c_0 \intdni g z^n \overline{S_1 }(u^n,v^n,\eta^n)dx \right\}+\Re\left\{2\intdni g f^{2,n}\overline{S_1 }(u^n,v^n,\eta^n)dx \right\}.
		\end{array}
		\end{equation*}
	Using the definition of $g$ and Cauchy-Schwarz inequality in  the above equations, then using the fact that \begin{equation*}
		\left\{\begin{array}{lll}  \displaystyle \intdni |z^n |^2 dx \leq \intdx |z^n |^2 dx \leq \|U^n \|_{\HH}^2 =1,\quad  \displaystyle \intdni |f^{1,n}_x |^2 dx \leq \intdx |f^{1,n}_x  |^2 dx \leq a^{-1}\|F^n \|_{\HH}^2   \vspace{0.25cm}\\ \text{and}\quad \displaystyle \intdni |f^{2,n}|^2 dx \leq  \intdx |f^{2,n}  |^2 dx \leq \|F^n \|_{\HH}^2, \end{array}\right.
		\end{equation*}we obtain 
	\begin{equation}\label{3.28k}	
	\begin{array}{lll}
	\displaystyle	|v^n (\beta )|^2 +|v^n (\alpha )|^2 \leq\displaystyle M_{g^{\prime}}\intdni |v^n |^2 dx +2|\la^n | M_{g}\left( \intdni |u^{n}_x |^2 dx\right)^{\frac{1}{2}}\left( \intdni |v^{n}  |^2 dx\right)^{\frac{1}{2}}\\
	\hspace{3.2cm}+\, \displaystyle \frac{2}{\sqrt{a}} M_{g}\left( \intdni |v^{n}  |^2 dx\right)^{\frac{1}{2}}\|F^n \|_{\HH}
	\end{array}\end{equation}and
	\begin{equation}\label{3.31k}
	\begin{array}{lll}
	\displaystyle	\left|\left(S_1(u^n,v^n,\eta^n) \right)(\beta^{-} )\right|^2 +\left|\left(S_1(u^n,v^n,\eta^n) \right)(\alpha )\right|^2 \leq  \displaystyle M_{g^{\prime}} \intdni |S_1 (u^n,v^n,\eta^n) |^2 dx \vspace{0.25cm}\\ 
	\displaystyle +\,2|\la^n |M_g \left(\intdni\left|S_1 (u^n,v^n,\eta^n)\right|^2 dx\right)^{\frac{1}{2}} \left(\intdni|v^{n} |^2 dx\right)^{\frac{1}{2}}\vspace{0.25cm}\\
\displaystyle +\,2c_0 M_g \left(\intdni|S_1 (u^n,v^n,\eta^n) |^2 dx\right)^{\frac{1}{2}}  + 2M_g  \left(\intdni|S_1 (u^n,v^n,\eta^n) |^2 dx\right)^{\frac{1}{2}}\|F^n \|_{\HH} .
	\end{array}
	\end{equation}
	Finally, passing to limit in  \eqref{3.28k} and \eqref{3.31k}, then using \eqref{contra1}, Lemma \ref{firstlemmass} and the fact that $\|F^n \|_{\HH}\to 0 $, we obtain \eqref{limvnboundary} and \eqref{limbigbig}. The proof is thus complete. 
\end{proof}
\begin{lem}{\rm\label{3rdlemmass}
		Under the  hypothesis \eqref{H}, the solution $U^n =(u^n ,v^n ,y^n ,z^n ,\eta^n (\cdot,\rho))^{\top}\in D(\AA )$ of  system \eqref{eq1ss}-\eqref{eq4ss} satisfies the following limits
		\begin{equation}\label{limznynx}
		\lim_{n \to \infty} \intdni |z^n |^2 dx =0\quad \text{and} \quad
		\lim_{n \to \infty} \intdni |y^{n}_x  |^2 dx =0.
		\end{equation}}
\end{lem}
\begin{proof}
	First, multiplying  \eqref{eq2ss} by $\overline{z^n }$,  integrating over $(\alpha ,\beta )$, using the definition of $c(\cdot)$ and $S_b (u^n,v^n,\eta^n)$, then taking the real part, we get 
	\begin{equation}\label{zm}
	\Re\left\{i\la^n \intdni v^n \overline{z^n }dx\right\}-\Re\left\{\intdni \left(S_{1} (u^n,v^n,\eta^n) \right)_{x}  \overline{z^n }dx\right\}+c_0 \intdni |z^n |^2 dx =\Re \left\{\intdni f^{2,n} \overline{z^n }dx\right\}.
	\end{equation}From  \eqref{eq3ss}, we deduce that 
	\begin{equation}\label{3.25new}
	\overline{z^{n}_x }=-i\la^n \overline{y^{n}_x }-\overline{f^{3,n}_x }. 
	\end{equation} Using integration by parts to the second term in 	 \eqref{zm}, then using \eqref{3.25new}, we get 
	\begin{equation*}\label{3.42}
	\begin{array}{lll}
	\displaystyle	c_0 \intdni |z^n |^2 dx =	\displaystyle \Re \left\{i\la^n \intdni S_1 (u^n,v^n,\eta^n)\overline{y^{n}_x }dx\right\}+\Re\left\{\intdni S_{1}(u^n,v^n,\eta^n) \overline{f^{3,n}_x }dx\right\}\qquad\vspace{0.25cm}\\
	\hspace{2.5cm}	\displaystyle +\,\Re\left\{\left[S_{1} \left(u^n,v^n,\eta^n\right) \overline{z^n} \, \right]_{\alpha }^{\beta }\right\}+\Re \left\{\intdni f^{2,n} \overline{z^n }dx\right\}-\Re\left\{i\la^n \intdni v^n \overline{z^n }dx\right\}.
	\end{array}
	\end{equation*}
	Using Cauchy-Schwarz inequality in the above equation   and the fact that $\displaystyle\intdni |\xi_{1}^n |^2 dx \leq \intdx |\xi_{1}^n |^2 dx \leq  \|U^n \|_{\HH}^2 =1$ with $\displaystyle \xi_{1}^n \in \{y^{n}_x , z^n \} $ and $\displaystyle\intdni |\xi_{2}^n |^2 dx \leq \intdx |\xi_{2}^n |^2 dx \leq  \|F^n \|_{\HH}^2 $ with $\displaystyle \xi_{2}^n \in \{f^{2,n} , f^{3,n}_x \} $, we obtain
	\begin{equation}\label{3.43ss}
	\begin{array}{lll}
	\displaystyle	c_0 \intdni |z^n |^2 dx \leq	\displaystyle \left(\left|\la^n \right|+\left\|F^n \right\|_{\HH}\right)\left(\intdni |S_{1} (u^n,v^n,\eta^n)|^2 dx\right)^{\frac{1}{2}}+|\la^n |\left(\intdni|v^n |^2 dx\right)^{\frac{1}{2}}+\|F^n \|_{\HH}\vspace{0.25cm}\\
	\hspace{2.5cm}	\displaystyle+\,\left|\left(S_{1}(u^n,v^n,\eta^n)\right) (\beta^- )\right||z^n (\beta )|+\left|\left(S_{1} (u^n,v^n,\eta^n)\right)(\alpha )\right||z^n (\alpha)|.
	\end{array}
	\end{equation}
	Now, using the fact that {\small$ \displaystyle \intdni |z^n |^2 dx \leq \intdx |z^n |^2 dx \leq \|U^n \|_{\HH}^2 =1$} in \eqref{zboundary}, we get
	\begin{equation}\label{bouznsimp}
	|z^n (x )|\leq \left(M_{g^{\prime}}+2|\la^n |M_{g}+2M_{g}\|F^n \|_{\HH}\right)^{\frac{1}{2}}\quad \text{for}\quad  x\in \{\alpha,\beta  \}.
	\end{equation}Inserting \eqref{bouznsimp} in \eqref{3.43ss}, we obtain
	\begin{equation*}\label{3.384}
	\begin{array}{lll}
	\displaystyle	c_0 \intdni |z^n |^2 dx \leq	\displaystyle (|\la^n |+\|F^n \|_{\HH})\left(\intdni |S_{1} (u^n,v^n,\eta^n)|^2 dx\right)^{\frac{1}{2}}+|\la^n |\left(\intdni|v^n |^2 dx\right)^{\frac{1}{2}}+\|F^n \|_{\HH},\vspace{0.25cm}\\
	\hspace{2.5cm}	\displaystyle+\, \left(M_{g^{\prime}}+2\left|\la^n \right|M_{g}+2M_{g}\left\|F^n \right\|_{\HH}\right)^{\frac{1}{2}}\left(\left|\left(S_{1}(u^n,v^n,\eta^n)\right) (\beta^- )\right|+\left|\left(S_{1}(u^n,v^n,\eta^n)\right) (\alpha )\right|\right).
	\end{array}
	\end{equation*}
	Therefore, by passing to the limit in the above inequality and by using \eqref{contra1}, \eqref{limbigbig}, Lemma \ref{firstlemmass} and the fact that $\displaystyle \|F^n \|_{\HH}\to 0$, we obtain  the first limit in \eqref{limznynx}.
	On the other hand, multiplying  \eqref{eq4ss} by $-\overline{z^n }(\la^n )^{-1}$, integrating over $(\alpha ,\beta )$, using the definition of $c(\cdot)$, then taking the imaginary part, we get
	\begin{equation*}\label{intynx}
		-\intdni |z^n |^2 dx +\Im\left\{ (\la^n )^{-1}\intdni y^{n}_{xx}\overline{z^n }dx\right\}+\Im\left\{ c_0 (\la^n )^{-1} \intdni v^n \overline{z^n }dx\right\}=-\Im\left\{(\la^n )^{-1}\intdni f^{4,n}\overline{z^n }dx \right\}.
		\end{equation*}
	Using integration by parts to the second term in the above equation, then using \eqref{3.25new}, we obtain
	\begin{equation*}\label{intynxsimplify}
	\begin{array}{lll}
	\displaystyle	\intdni |y^{n}_x |^2 dx =\displaystyle\intdni |z^n |^2 dx -\Im\left\{(\la^n )^{-1}\intdni  \overline{f^{3,n}_x }y^{n}_x dx\right\}-\Im \left\{(\la^n )^{-1}\left[y^{n}_x \overline{z^n }\right]_{\alpha }^{\beta }\right\} \vspace{0.25cm}\\\hspace{2cm}\displaystyle-\,\Im\left\{c_0 (\la^n )^{-1}\intdni v^n \overline{z^n }dx\right\}-\Im\left\{(\la^n )^{-1}\intdni f^{4,n}\overline{z^n }dx\right\}.
	\end{array}
	\end{equation*}
	Using Cauchy-Schwarz inequality in  the above equation and the fact that $\|U^n \|_{\HH}=1$, we get 
	\begin{equation}\label{3.43}
	\begin{array}{lll}
	\displaystyle	\intdni |y^{n}_x |^2 dx \leq	\displaystyle\intdni |z^n |^2 dx  +c_0 |\la^n |^{-1}\left(\intdni|v^n |^2 dx\right)^{\frac{1}{2}}+2|\la^n |^{-1}\|F^n \|_{\HH}\vspace{0.25cm}\\\hspace{2.2cm}+\,	\displaystyle|\la^n |^{-1}|y^{n}_x (\beta )||z^{n} (\beta )|+|\la|^{-1}|y^{n}_x (\alpha)||z^{n} (\alpha)|.
	\end{array}	\end{equation}
	Moreover, using  the fact that {\small $\displaystyle \intdni |y^{n}_x |^2 dx  \leq \intdx |y^{n}_x |^2 dx \leq \|U^n \|_{\HH}^2 =1 $} in \eqref{ynxboundary}, we get 
	\begin{equation}\label{3.45}
	|y^{n}_x (x )|\leq \left(M_{g^{\prime}}+2(|\la^n |+c_0 )M_{g}+2M_{g}\|F^n \|_{\HH}\right)^{\frac{1}{2}}\quad	\text{for} \quad x \in \{\alpha , \beta\} . \end{equation}Inserting \eqref{3.45} in \eqref{3.43}, we obtain 	\begin{equation}\label{3.43new}
	\begin{array}{lll}
	\displaystyle	\intdni |y^{n}_x |^2 dx \leq	\displaystyle\intdni |z^n |^2 dx  +c_0 |\la^n |^{-1}\left(\intdni|v^n |^2 dx\right)^{\frac{1}{2}}+2|\la^n |^{-1}\|F^n \|_{\HH}\vspace{0.25cm}\\\hspace{2cm}+\,	\displaystyle|\la^n |^{-1}\left(M_{g^{\prime}}+2(|\la^n |+c_0 )M_{g}+2M_{g}\|F^n \|_{\HH}\right)^{\frac{1}{2}}\left( |z^{n} (\beta )|+|z^{n} (\alpha )|\right).
	\end{array}	\end{equation}
	Now, passing to the limit in inequality \eqref{zboundary}, then using \eqref{contra1}, the first limit in \eqref{limznynx} and the fact that $\|F^n \|_{\HH}\to 0$, we get \begin{equation}\label{3.44}
	\lim_{n\to \infty}\left( \left|z^n (\beta )\right|^2+ \left|z^n (\alpha)\right|^2 \right)=0. 
	\end{equation}
	Finally, passing to the limit in \eqref{3.43new}, then using \eqref{contra1}, \eqref{limvn}, the first limit in  \eqref{limznynx}, \eqref{3.44} and the fact that $\|F^n \|_{\HH}\to 0$, we obtain the second limit in  \eqref{limznynx}. The proof is thus complete.
\end{proof}
\begin{lem}\label{matrix}
	{\rm Under the  hypothesis \eqref{H}, the solution $U^n =(u^n ,v^n ,y^n ,z^n ,\eta^n (\cdot,\rho))^{\top}\in D(\AA )$ of system \eqref{eq1ss}-\eqref{eq5ss} satisfies the following estimations
		\begin{equation}\label{gonf1}
		\lim_{n\to \infty}|u^n (\beta )|^2 =0 \quad \text{and } \quad \lim_{n\to \infty}|y^n (\beta )|^2 =0,
		\end{equation}
		\begin{equation}\label{gonf2}
		\lim_{n\to \infty}|u^{n}_x (\beta^{+} )|^2 =0 \quad \text{and } \quad \lim_{n\to \infty}|y^{n}_x (\beta )|^2 =0,
		\end{equation}
		\begin{equation}\label{gonf3}
		\lim_{n\to \infty }\left(\int_{\beta }^{\gamma } |u^n |^2 dx+\int_{\beta }^{\gamma } |u^{n}_x |^2 dx+\int_{\beta }^{\gamma } |y^n |^2 dx +\int_{\beta }^{\gamma } |y^{n}_x |^2 dx\right)=0,
		\end{equation}
		\begin{equation}\label{gonf4}
		\lim_{n\to \infty}	\int_{\beta }^{\gamma } |v^n |^2 dx=0 \quad \text{and}\quad \lim_{n\to \infty }\int_{\beta }^{\gamma } |z^n |^2 dx =0 . 
		\end{equation}}
\end{lem}
\begin{proof} First, from  \eqref{eq1ss} and \eqref{eq3ss}, we get 
	\begin{equation*}
	|u^n (\beta )|^2 \leq 2(\la^n )^{-2}|v^n (\beta )|^2 +2(\la^n )^{-2}|f^{1,n}(\beta )|^2
	\end{equation*}
	and 
	\begin{equation*}
	|y^n (\beta )|^2 \leq 2(\la^n )^{-2}|z^n (\beta )|^2 +2(\la^n )^{-2}|f^{3,n}(\beta )|^2.
	\end{equation*}
Using the fact that  {\small$\displaystyle|f^{1,n}(\beta )|^2 \leq \beta \int_{0}^{\beta }|f^{1,n}_x |^2 dx \leq \beta a^{-1}  \|F^n \|_{\HH}^2 $ and  $\displaystyle|f^{3,n}(\beta )|^2 \leq \beta \int_{0}^{\beta }|f^{3,n}_x |^2 dx \leq \beta  \|F^n \|_{\HH}^2 $} in  the above inequalities, we obtain
	\begin{equation*}\label{unalph3}
	|u^n (\beta )|^2 \leq 2(\la^n )^{-2}|v^n (\beta )|^2 +2\beta a^{-1} (\la^n )^{-2}\|F^n \|_{\HH}^2 \end{equation*}and
	\begin{equation*}\label{ynalph3}	|y^n (\beta )|^2 \leq 2(\la^n )^{-2}|z^n (\beta )|^2 +2\beta  (\la^n )^{-2}\|F^n \|_{\HH}^2 .
	\end{equation*}
	 Passing to the limit in the above inequalities, then using \eqref{contra1}, \eqref{limvnboundary}, \eqref{3.44} and the fact that $\|F^n \|_{\HH}\to 0$, we obtain \eqref{gonf1}.
	Second, since $S_b (u^n,v^n,\eta^n)\in H^1 (0,L)\subset C([0,L])$, then we deduce that
	\begin{equation}\label{modulssn}
	\left|\left(S_{1} (u^n,v^n,\eta^n)\right) (\beta^{-})\right|^2 =|a u^{n}_x (\beta^{+})|^2 .
	\end{equation}Thus, from \eqref{limbigbig} and \eqref{modulssn}, we obtain the first limit in \eqref{gonf2}. Moreover, passing to the limit in inequality \eqref{ynxboundary}, then using \eqref{contra1}, the second limit in \eqref{limznynx} and the fact that $\|F^n \|_{\HH}\to 0$, we obtain the second limit in \eqref{gonf2}. 
	On the other hand,  \eqref{eq1ss}-\eqref{eq4ss} can be written in $(\beta ,\gamma )$ as the following form
	\begin{eqnarray}
	(\la^n )^2 u^n +au^{n}_{xx}-i\la^n c_0 y^n &=&G^{1,n} \quad \text{in} \  \ (\beta,\gamma),\label{3.58}\\(\la^n )^2 y^n +y^{n}_{xx}+i\la^n c_0 u^n &=&G^{2,n}\quad \text{in}\ \  (\beta,\gamma),\label{3.59}
	\end{eqnarray}where \begin{equation}\label{g1ng2n}\displaystyle G^{1,n}=-f^{2,n}-i\la^n f^{1,n}-c_0 f^{3,n} \quad \text{and} \quad  \displaystyle G^{2,n}=-f^{4,n}-i\la^n f^{3,n}+c_0 f^{1,n}. \end{equation}
	Let $V^n =(u^n ,u^{n}_x ,y^n ,y^{n}_x )^{\top}$, then  \eqref{3.58}-\eqref{3.59} can be written as the following \begin{equation}\label{de}
	V^{n}_x =B^n V^n +G^n ,
	\end{equation}where
	\begin{equation*}
	B^n =  \begin{pmatrix}
	0&1&0&0\\
	-a^{-1}(\la^n )^2 &0&a^{-1}i\la^n c_0 &0\\
	0&0&0&1\\
	-i\la^n c_0 &0&-(\la^n )^2 &0
	\end{pmatrix}=(b_{ij})_{1\leq i,j\leq 4}\quad
	\text{and}\quad 
	G^n =\begin{pmatrix}
	0 \\
	a^{-1}G^{1,n}\\
	0\\
	G^{2,n}
	\end{pmatrix}.
	\end{equation*}The solution of the differential equation \eqref{de} is given by
	\begin{equation}\label{solde}
	V^{n}(x)=e^{B^n (x-\beta )}V^n (\beta^+ )+\int_{\beta }^{x}e^{B^n (s-x)}G^n (s)ds,
	\end{equation}where $\displaystyle e^{B^n (x-\beta )}=(c_{ij})_{1\leq i,j\leq 4}$ and $\displaystyle e^{B^n (s-x )}=(d_{ij})_{1\leq i,j \leq 4}$ are denoted by the exponential of the matrices $B^n (x-\beta )$ and $B^n (s-x)$ respectively. Now, from \eqref{contra1}, the entries $b_{ij}$ are bounded for all $1\leq i,j \leq 4$ and consequently, the entries $b_{ij}$ $(x-\beta )$ and $b_{ij}$ $(s-x )$ are bounded. 
	In addition, from the definition of the exponential of a square matrix, we obtain
	\begin{equation*}
	e^{B^n \zeta}=\sum_{k=0}^{\infty} \frac{(B^n \zeta)^k }{k!}\quad \text{for} \quad \zeta \in \{x-\beta, s-x\} .
	\end{equation*} Therefore, the entries $c_{ij}$ and $d_{ij}$ are also bounded for all $1\leq i,j\leq 4$ and consequently, $e^{B^n (x-\beta )}$ and $e^{B^n (s-x )}$ are two bounded matrices.
	From \eqref{gonf1} and \eqref{gonf2}, we directly obtain \begin{equation}\label{ V(alph3)} V^n (\beta^+ ) \to 0 \quad \text{in} \quad (L^2 (\beta ,\gamma ))^4 , \quad \text{as} \quad n\to \infty. \end{equation}
	Moreover, from \eqref{g1ng2n}, we deduce that \begin{equation}\label{G1n}
	\int_{\beta}^{\gamma} |G^{1,n}|^2 dx \leq 3\intdx |f^{2,n} |^2 dx +3 (\la^n )^2 \intdx|f^{1,n}|^2 dx +3c_{0}^2  \intdx|f^{3,n}|^2 dx 
	\end{equation}and \begin{equation}\label{G2n}
	\int_{\beta}^{\gamma} |G^{2,n}|^2 dx \leq 3\intdx |f^{4,n} |^2 dx +3 (\la^n )^2 \intdx|f^{3,n}|^2 dx +3c_{0}^2  \intdx|f^{1,n}|^2 dx .
	\end{equation}Now, since $f^{1,n}, f^{3,n} \in H^{1}_0 (0,L)$, then it follows by Poincar\'e inequality that there exist  two constants $C_1 >0$ and $C_2 >0 $ such that 
	\begin{equation}\label{poincare}
	\|f^{1,n}\|_{L^2 (0,L)} \leq C_{1} \|f^{1,n}_x \|_{L^2 (0,L)} \quad \text{and}\quad 	\|f^{3,n}\|_{L^2 (0,L)} \leq C_{2}\|f^{3,n}_x \|_{L^2 (0,L)}.
	\end{equation}Consequently, from \eqref{G1n}, \eqref{G2n} and \eqref{poincare}, we get
	\begin{equation}\label{normg1n}
	\int_{\beta}^{\gamma}|G^{1,n}|^2 dx \leq 3\left(1+a^{-1}(\la^n  C_{1})^2 +(c_{0} C_{2})^2 \right)\|F^n \|_{\HH}^2 ,
	\end{equation} and 
	\begin{equation}\label{normg2n}
	\int_{\beta}^{\gamma}|G^{2,n}|^2 dx \leq 3\left(1+(\la^n  C_{1})^2 +a^{-1}(c_{0} C_{2})^2 \right)\|F^n \|_{\HH}^2 .
	\end{equation}
	Hence, from \eqref{contra1}, \eqref{normg1n}, \eqref{normg2n}  and the fact that $\|F^n \|_{\HH}\to 0$, we obtain \begin{equation}\label{Gn} G^n \to 0\quad \text{in}\quad  (L^2 (\beta ,\gamma ))^4 , \quad \text{as} \quad n\to \infty. \end{equation} Therefore, from \eqref{solde}, \eqref{ V(alph3)}, \eqref{Gn} and as $e^{B^n (x-\beta )}$, $e^{B^n (s-x )}$ are two bounded matrices, we get $V^n \to 0 $ in $  (L^2 (\beta ,\gamma))^4$ and consequently, we obtain \eqref{gonf3}. Next, from \eqref{eq1ss} , \eqref{eq3ss} and \eqref{poincare}, we deduce that 
	\begin{eqnarray*}
	&&	\int_{\beta}^{\gamma }|v^n |^2 dx \leq  2(\la^n )^2 \int_{\beta }^{\gamma }|u^n |^2 dx +2\int_{\beta }^{\gamma }|f^{1,n}|^2 dx\leq  2(\la^n )^2 \int_{\beta }^{\gamma }|u^n |^2 dx +2C_1 a^{-1}\|F^n \|_{\HH}^2  ,\label{3.70}\\&&	\int_{\beta}^{\gamma }|z^n |^2 dx \leq 2(\la^n )^2 \int_{\beta }^{\gamma }|y^n |^2 dx +2\int_{\beta }^{\gamma }|f^{3,n}|^2 dx\leq 2(\la^n )^2 \int_{\beta }^{\gamma }|y^n |^2 dx +2C_2 \|F^n \|_{\HH}^2.\label{3.71}
	\end{eqnarray*}
	Finally, passing to the limit in the above inequalities, then using \eqref{contra1}, \eqref{gonf3} and the fact that $\|F^n \|_{\HH}\to 0$, we obtain \eqref{gonf4}. The proof is thus complete.
\end{proof}
\begin{lem}\label{2hsn2hynx}{\rm Let $h \in C^1 ([0,L])$ be a function. Under the  hypothesis \eqref{H}, the solution $U^n =(u^n ,v^n ,y^n ,z^n ,\eta^n (\cdot,\rho) )^{\top}\in D(\AA )$ of  system \eqref{eq1ss}-\eqref{eq5ss} satisfies the following estimation}
		\begin{equation*}\label{2hsn2hynxeq}
		\begin{array}{lll}
		&& \displaystyle\intdx h^{\prime}\left(a^{-1}|S_b (u^n,v^n,\eta^n) |^2 +|v^n |^2 +|z^n |^2 + \left|y^{n}_x \right|^2\right)dx-\left[h\left(a^{-1}|S_b(u^n,v^n,\eta^n) |^2 +\left|y^{n}_x \right|^2 \right)\right]_{0}^{L}\vspace{0.25cm}\\
		&&\displaystyle -\,\Re\left\{2\intdx c(\cdot)h v^n \overline{y^{n}_x  }dx\right\}  +\Re\left\{ \frac{2}{a}\intdx c(\cdot)h z^n \overline{S_b}(u^n,v^n,\eta^n) dx\right\}+\Re \left\{\frac{2i\la^n}{a} \intdx b(\cdot)hv^n (\kappa_{1}\overline{v^{n}_x }+\kappa_2 \overline{\eta^{n}_x }(\cdot,1))dx\right\}\vspace{0.25cm}\\&&=\displaystyle \Re\left\{2\intdx h\overline{f^{1,n}_x }v^n dx \right\}+\Re \left\{\frac{2}{a}\intdx h f^{2,n}\overline{S_b  }(u^n,v^n,\eta^n)dx \right\}+\Re \left\{2\intdx h \overline {f^{3,n}_x } z^ndx \right\}+\Re \left\{2\intdx h f^{4,n}\overline{y^{n}_x  }dx \right\}.
		\end{array}
		\end{equation*}
\end{lem}
\begin{proof}
	First, multiplying  \eqref{eq2ss} and \eqref{eq4ss} by $2a^{-1}h\overline{S_b }(u^n,v^n,\eta^n)$ and $2h\overline{y^{n}_x }$ respectively, integrating over $(0,L)$, then taking the real part, we get 
		\begin{equation}\label{inthsn}
		\begin{array}{lll}
	\displaystyle 	\Re\left\{\frac{2i\la^n}{a}  \intdx h v^n \overline{S_b }(u^n,v^n,\eta^n)dx\right\}-a^{-1}\intdx h\left(\left|S_b (u^n,v^n,\eta^n)  \right|^2 \right)_x dx \vspace{0.25cm}\\
		\displaystyle +\,\Re\left\{\frac{2}{a}\intdx c(\cdot)hz^n \overline{
			S_b  }(u^n,v^n,\eta^n)dx\right\}=\Re\left\{ \frac{2}{a}\intdx hf^{2,n}\overline{S_b  }(u^n,v^n,\eta^n)dx \right\}
		\end{array}
		\end{equation}
		and
		\begin{equation}\label{inthynx}
	\Re\left\{ 2i\la^n \intdx h z^n \overline{y^{n}_x  }dx\right\}-\intdx h\left(\left|y^{n}_x  \right|^2 \right)_x dx -\Re\left\{2\intdx c(\cdot)hv^n \overline{
		y^{n}_x }dx\right\}=\Re\left\{2\intdx hf^{4,n}\overline{y^{n}_x  }dx \right\}.
	\end{equation}
	From  \eqref{eq1ss} and \eqref{eq3ss}, we deduce that 
	\begin{eqnarray}
	i\la^n \overline{u^{n}_x }=-\overline{v^{n}_x }-\overline{f^{1,n}_x },\label{1barx}\\
	i\la^n \overline{y^{n}_x }=-\overline{z^{n}_x }-\overline{f^{3,n}_x }\label{3barx}.
	\end{eqnarray}
	Consequently, from  \eqref{1barx} and the definition $S_b (u^n,v^n,\eta^n)$, we have
	\begin{equation}\label{ilasbar}
	i\la \overline{S_b }(u^n,v^n,\eta^n)=-a\left(\overline{v^{n}_x } +\overline{f^{1,n}_x }\right)+i\la b(\cdot)\left(\kappa_1 \overline{v^{n}_x }+\kappa_2 \overline{\eta^{n}_x } (\cdot,1)\right).
	\end{equation}
	Substituting \eqref{ilasbar} and \eqref{3barx} in  \eqref{inthsn} and \eqref{inthynx} respectively, we obtain
\begin{equation*}\label{3.77}
	\begin{array}{lll}
		&& \displaystyle-\,\intdx h\left(\left|v^n \right|^2 +a^{-1}\left|S_b(u^n,v^n,\eta^n)\right|^2 \right)_x dx +\Re \left\{\frac{2i\la^n}{a} \intdx b(\cdot)hv^n (\kappa_{1}\overline{v^{n}_x }+\kappa_2 \overline{\eta^{n}_x }(\cdot,1))dx\right\}\vspace{0.25cm}\\&&\displaystyle+\,\Re\left\{\frac{2}{a}\intdx c(\cdot)hz^n \overline{
			S_b  }(u^n,v^n,\eta^n)dx\right\} =\displaystyle \Re\left\{2\intdx h\overline{f^{1,n}_x } v^n dx\right\}+\Re\left\{ \frac{2}{a}\intdx hf^{2,n}\overline{S_b  }(u^n,v^n,\eta^n)dx \right\}\vspace{0.25cm}
		\end{array}
		\end{equation*} 
		and
\begin{equation*}\label{inthynxsub}
		\begin{array}{lll}
		&&\displaystyle -\,\intdx h \left(\left|z^n \right|^2 +\left|y^{n}_x \right|^2  \right)_x dx -\Re\left\{2\intdx c(\cdot)hv^n \overline{
			y^{n}_x }dx\right\}=	\displaystyle \Re\left\{2\intdx hf^{4,n}\overline{y^{n}_x  }dx \right\}+	\displaystyle \Re\left\{2\intdx h \overline{f^{3,n}_x }z^n dx  \right\}.
		\end{array}
		\end{equation*}
		Finally,  adding the above equations, then using integration by parts and the fact that $v^n (0)=v^n (L)=0$ and $z^n (0)=z^n (L)=0$, we obtain  the desired result. The proof is thus complete.
\end{proof}\\\linebreak
Now, we fix the cut-off functions $\chi_1 , \chi_2\in C^1 ([0,L])$ (see Figure \ref{p1-Fig2}) such that $0\leq \chi_1(x)\leq 1$, $0\leq \chi_2(x)\leq 1$, for all $x\in[0,L]$ and
\begin{equation*}
\chi_1 (x)= 	\left \{ \begin{array}{lll}
1 \quad \text{if } \quad x \in [0,\alpha],\\
0 \quad \text{if } \quad x \in [\beta ,L],
\end{array}	\right. \,\text{and}\quad
\chi_2 (x) = 	\left \{ \begin{array}{lll}
0\quad \text{if } \quad x \in [0,\beta],\\
1 \quad \text{if } \quad x \in [\gamma ,L],
\end{array}	\right.
\end{equation*}
and set
$
\displaystyle\max_{x\in [0 ,L]}|\chi_1^{\prime} (x)|= M_{\chi_1^{\prime}}\quad \text{and}\quad  \displaystyle\max_{x\in [0,L]}|\chi_2^{\prime} (x)|= M_{\chi_2^{\prime}},\qquad\quad\,\,
$
\begin{figure}[h!]	
	\begin{center}
		\begin{tikzpicture}
		\draw[->](0,0)--(9,0);
		\draw[->](0,0)--(0,3);
		
		
		\node[black,below] at (0,0){\scalebox{0.75}{$0$}};
		\node at (0,0) [circle, scale=0.3, draw=black!80,fill=black!80] {};
		
		\node[black,below] at (2,0){\scalebox{0.75}{$\alpha$}};
		\node at (2,0) [circle, scale=0.3, draw=black!80,fill=black!80] {};
		
		\node[black,below] at (4,0){\scalebox{0.75}{$\beta$}};
		\node at (4,0) [circle, scale=0.3, draw=black!80,fill=black!80] {};

		\node[black,below] at (6,0){\scalebox{0.75}{$\gamma$}};
		\node at (6,0) [circle, scale=0.3, draw=black!80,fill=black!80] {};

		\node[black,below] at (8,0){\scalebox{0.75}{$L$}};
		\node at (8,0) [circle, scale=0.3, draw=black!80,fill=black!80] {};
		
		
		\node at (0,2) [circle, scale=0.3, draw=black!80,fill=black!80] {};
		
		\node[black,left] at (0,2){\scalebox{0.75}{$1$}};
		
		\node[black,right] at (9.5,3){\scalebox{0.75}{$\chi_1$}};
		\node[black,right] at (9.5,2.5){\scalebox{0.75}{$\chi_2$}};

		\draw[-,red](0,2)--(2,2);
		\draw [red] (2,2) to[out=0.40,in=180] (4,0) ;
		\draw[-,red](4,0)--(8,0);
		
		\draw[-,blue](0,0)--(4,0);
		\draw [blue] (4,0) to[out=0.40,in=180] (6,2) ;
		\draw[-,blue](6,2)--(8,2);
		
	
		
		
		\draw[-,red](9,3)--(9.5,3);
		\draw[-,blue](9,2.5)--(9.5,2.5);
		\end{tikzpicture}
	\end{center}
	\caption{Geometric description of the functions $\chi_1$ and $\chi_2$.}\label{p1-Fig2}
\end{figure}
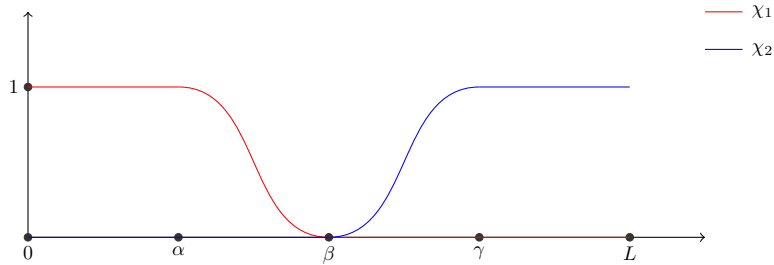
\begin{lem}\label{intcutoff1}{\rm
		Under the  hypothesis \eqref{H}, the solution $U^n =(u^n ,v^n ,y^n ,z^n ,\eta^n (\cdot,\rho) )^{\top}\in D(\AA )$ of  system \eqref{eq1ss}-\eqref{eq5ss} satisfies the following limits} \begin{eqnarray}
	\lim_{n\to \infty }\left(\int_{0}^{\alpha }|y^{n}_x |^2 dx +\int_{0}^{\alpha }|z^n  |^2 dx \right)=0,\label{intcutoff1eq}\\
	\lim_{n\to \infty }\left(a \int_{\gamma }^{L }|u^{n}_x |^2 dx + \int_{\gamma }^{L }|v^{n} |^2 dx +\int_{\gamma }^{L }|y^{n}_x |^2 dx +\int_{\gamma}^{L }|z^n   |^2 dx \right)=0.\label{intcutoff2}
	\end{eqnarray}
\end{lem}
\begin{proof}
First, using the result of Lemma \ref{2hsn2hynx} with $h=x \chi_1 $, then using the definition of $b(\cdot)$, $c(\cdot)$, $S_b (u^n,v^n,\eta^n)$ and $\chi_1 $, we get
	\begin{equation*}
		\begin{array}{lll}
			\displaystyle  \int_{0}^{\alpha }|y^{n}_x |^2 dx+ \int_{0}^{\alpha }|z^n  |^2 dx=-\int_{0}^{\alpha }|v^{n} |^2 dx -a^{-1}\int_{0 }^{\alpha } |S_1 (u^n,v^n,\eta^n) |^2  dx\vspace{0.25cm}\\ 
		\displaystyle  -\int_{\alpha }^{\beta }\left(\chi_1 +x\chi_1^{\prime}\right)\left(a^{-1}\left|S_1(u^n,v^n,\eta^n) \right|^2 +|v^n |^2+|y^{n}_x |^2 +|z^n  |^2 \right)dx-\Re\left\{\frac{2c_0}{a} \int_{\alpha }^{\beta }  x\chi_1  z^n \overline{S_1 }(u^n,v^n,\eta^n)dx\right\}\vspace{0.25cm}\\
		\displaystyle +\,\Re\left\{2c_0\int_{\alpha}^{\beta}x\chi_1 v^n \overline{y^n_x}dx\right\}-  \Re \left\{ \frac{2i\la^n}{a}  \int_{0}^{\beta }x\chi_1 v^n \left(\kappa_1 \overline{v^{n}_x }+\kappa_2 \overline{\eta^{n}_x }(\cdot,1) \right)dx\right\}+\Re\left\{ \frac{2}{a}\int_{0}^{\beta }x\chi_1f^{2,n}\left(\kappa_1 \overline{v^{n}_x }+\kappa_2 \overline{\eta^{n}_x }(\cdot,1) \right)dx\right\} \vspace{0.25cm}\\
		\displaystyle +\,\Re\left\{2\intdx x\chi_1 \left(\overline{f^{1,n}_x }v^n +f^{2,n} \overline{u^{n}_x }+\overline{f^{3,n}_x }z^n +f^{4,n} \overline{y^{n}_x} \right)dx \right\}.
		\end{array}
		\end{equation*}Using Cauchy-Schwarz inequality in the above Equation and  the fact that $\|U^n \|_{\HH}=\|(u^n,v^n,y^n,z^n,\eta^n(\cdot,\rho))^{\top} \|_{\HH}=1$, we obtain
	\begin{equation*} \begin{array}{lll}
	\displaystyle	\int_{0}^{\alpha }|y^{n}_x  |^2 dx+\int_{0}^{\alpha }|z^n |^2 dx \leq \int_{0}^{\alpha}|v^n|^2dx+a^{-1}\int_{0}^{\alpha}|S_1(u^n,v^n,\eta^n)|^2dx \vspace{0.25cm} \\
	  \displaystyle +\, \left(1+ \beta M_{\chi_1^{\prime}}\right)\int_{\alpha }^{\beta }\left(a^{-1}|S_1 (u^n,v^n,\eta^n)|^2 +|v^n |^2 +|z^n |^2 +\left|y^{n}_x \right|^2 \right)dx +\frac{2c_0 \beta}{a}  \left( \int_{\alpha }^{\beta }|S_1(u^n,v^n,\eta^n) |^2 dx \right)^{\frac{1}{2}}\vspace{0.25cm}\\
	  \displaystyle +\,2c_0 \beta \left( \int_{\alpha }^{\beta }|v^n |^2 dx \right)^{\frac{1}{2}}+\frac{2\beta}{a} \left(|\la^n |+\|F^n \|_{\HH}\right) \left[\kappa_1 \left( \int_{0 }^{\beta } |v^{n}_x |^2 dx \right)^{\frac{1}{2}} +|\kappa_2 | \left( \int_{0 }^{\beta } |\eta^{n}_x (\cdot ,1 ) |^2 dx \right)^{\frac{1}{2}}\right]\vspace{0.25cm}\\ 
	  \displaystyle +\,4L \left(\frac{1}{\sqrt{a}}+1\right)\|F^n \|_{\HH}.
	\end{array}\end{equation*}Therefore, by passing to the limit in the above inequality and by using \eqref{contra1}, Lemmas \ref{firstlemmass}, \ref{3rdlemmass} and the fact that $\|F^n \|_{\HH} \to 0$, we obtain \eqref{intcutoff1eq}.
	On the other hand, using the result of Lemma \ref{2hsn2hynx} with $h=(x-L)\chi_2 $, then using Cauchy-Schwarz inequality and the fact that $\|U^n \|_{\HH}=1$, we get 
	\begin{equation*}\label{cutoffchi2ineq}
	\begin{array}{lll}
\hspace{0.75cm}	\displaystyle	a\int_{\gamma }^L |u^{n}_x |^2 +	\int_{\gamma }^L|v^n |^2 dx +	\int_{\gamma }^L |y^{n}_x |^2 dx +	\int_{\gamma }^L |z^n |^2 dx \vspace{0.25cm}\\\hspace{1cm}
	\leq \displaystyle \left(1+(L-\beta)M_{\chi_2^{\prime}}\right)\int_{\beta }^{\gamma } \left( a|u^{n}_x |^2 +|v^n |^2 +|y^{n}_x |^2 +\left|z^n \right|^2 \right)dx+2c_0 (L-\beta ) \left( \int_{\beta }^{\gamma }|v^n |^2 dx \right)^{\frac{1}{2}}\left( \int_{\beta }^{\gamma }|y^{n}_x |^2 dx \right)^{\frac{1}{2}}\vspace{0.25cm}\\\hspace{1.25cm}
	 \displaystyle +\, 2c_0 (L-\beta ) \left( \int_{\beta }^{\gamma }|z^n |^2 dx \right)^{\frac{1}{2}}\left( \int_{\beta }^{\gamma }|u^{n}_x |^2 dx \right)^{\frac{1}{2}}+4L\left({\frac{1}{\sqrt{a}}}+1\right)\|F^n \|_{\HH}.
	\end{array}
	\end{equation*}Finally, passing to the limit in the above inequality, then using Lemma \ref{matrix} and the fact that $\|F^n \|_{\HH} \to 0$, we obtain \eqref{intcutoff2}. The proof is thus complete.
\end{proof}\\\linebreak
\noindent \textbf{Proof of Proposition \ref{pol-prop1}.} From Lemmas \ref{firstlemmass}-\ref{intcutoff1}, we obtain $\|U^n\|_{\HH}\to 0$ as $n\to\infty$ which contradicts $\|U^n\|_{\mathcal{H}}=1$. Thus, \eqref{Condition-1-pol} is holds true. The proof is thus complete.\xqed{$\square$}
\\\linebreak
\noindent \textbf{Proof of Theorem \ref{p1-strongthm2}.} From proposition \ref{pol-prop1}, we have $i\R\subset \rho (\AA)$ and consequently $\sigma (\AA)\cap i\R=\emptyset$. Therefore, according to Theorem \ref{App-Theorem-A.2}, we get that the $C_0-$semigroup of contraction $(e^{t\AA})_{t\geq0}$ is strongly stable. The proof is thus complete.\xqed{$\square$}

\section{Polynomial Stability}\label{Pol-Sta-Sec}
\noindent In this section, we will prove the polynomial stability of  system \eqref{sysorg0}-\eqref{initialcon}. The main result of this section is the following theorem. 
\begin{Thm}\label{polthm}{\rm
Under the  hypothesis \eqref{H}, for all $U_0\in D(\mathcal{A})$, there exists a constant $C>0$ independent of $U_0$ such that the energy of system \eqref{sysorg0}-\eqref{initialcon} satisfies the following estimation 
$$
E(t)\leq \frac{C}{t}\|U_0\|^2_{D(\AA)},\quad \forall\, t>0.
$$}
\end{Thm}
\noindent According to Theorem \ref{bt}, to prove Theorem \ref{polthm}, we will prove  the following two conditions
\begin{equation}\label{4.1}
i\R \subset \rho (\AA)
\end{equation}and
\begin{equation}\label{Condition-2-pol}
\sup_{\la\in\R}\frac{1}{|\la|^2 }\left\|(i\la I-\mathcal{A})^{-1}\right\|_{\mathcal{L}(\mathcal{\HH})}<+\infty.
\end{equation}

\noindent From proposition \ref{pol-prop1}, we obtain condition \eqref{4.1}. Next, we will prove condition \eqref{Condition-2-pol} by a contradiction argument. For this purpose,
suppose that \eqref{Condition-2-pol} is false, then there exists $\left\{(\la^n,U^n:=(u^n,v^n,y^n,z^n,\eta^n(\cdot,\rho))^{\top})\right\}_{n\geq1}\subset \R^{\ast}\times D(\mathcal{A})$ with
\begin{equation}\label{contra-pol2}
|\la^n|\to\infty \quad \text{and}\quad \|U^n\|_{\mathcal{H}}=\left\|(u^n ,v^n ,y^n ,z^n , \eta^n (\cdot,\rho))^{\top}\right\|_{\HH}=1,
\end{equation}
such that  
\begin{equation}\label{eq0ps}
	(\la^n )^2 (i\la^n I-\AA )U^n =F^n:=(f^{1,n},f^{2,n},f^{3,n},f^{4,n},f^{5,n}(\cdot,\rho))^{\top}  \to 0  \quad \text{in}\quad \HH.
\end{equation} 
For simplicity, we drop the index $n$. Equivalently, from \eqref{eq0ps}, we have 
\begin{eqnarray}
i\la u-v&=&\la^{-2}f^1  \to 0 \quad\qquad \text{in}\quad H^{1}_0 (0,L),\label{eq1ps}\\
i\la v-(S_b (u,v,\eta))_x   +c(\cdot)z&=&\la^{-2}f^2 \to 0 \quad\qquad \text{in}\quad L^2(0,L),\label{eq2ps}\\
i\la y-z&=&\la^{-2} f^3 \to 0 \quad\qquad \text{in}\quad H^{1}_0 (0,L),\label{eq3ps}\\
i\la z-y_{xx}-c(\cdot)v&=&\la^{-2}f^4  \to 0 \quad\qquad \text{in}\quad L^2 (0,L),\label{eq4ps}\\
i\la \eta (.,\rho)+\tau^{-1}\eta_{\rho}(.,\rho)&=&\la^{-2}f^5 (.,\rho) \to 0 \quad \text{in}\quad  \mathcal{W}.\label{eq5ps}
\end{eqnarray}
Here we will check the condition \eqref{Condition-2-pol} by finding a contradiction with \eqref{contra-pol2} such as $\left\|U\right\|_{\HH}=o(1)$.  For clarity, we divide the proof into several Lemmas.

\begin{lem}{\rm\label{firstlemmaps}
Under the  hypothesis \eqref{H},  the solution $U =(u ,v ,y ,z ,\eta (\cdot,\rho) )^{\top}\in D(\AA )$ of  system \eqref{eq1ps}-\eqref{eq5ps} satisfies the following estimations
	\begin{eqnarray}
 \intdnb |v_x |^2 dx =o(\la^{-2}), \label{vnxo}\\
 \intdnb |u_x |^2 dx =o(\la^{-4}), \label{unxo}\\
 	 \intdnb \int_{0}^{1} |\eta_x (\cdot,\rho) |^2 d\rho dx =o(\la^{-2}), \label{etao}\\
 \intdnb  |\eta_x (\cdot,1)|^2  dx =o(\la^{-2}),\label{eta1o}\\
 \intdnb |S_1(u,v,\eta) |^2 dx =o(\la^{-2}).\label{So}
	\end{eqnarray}}\end{lem}
\begin{proof} 	First, taking the inner product of \eqref{eq0ps} with $U$ in $\HH$ and  using \eqref{dissiparive} with the help of hypothesis \eqref{H}, we obtain
		\begin{equation}\label{4.12}\intdnb |v_x |^2 dx \leq -\frac{1}{\kappa_1 -|\kappa_2 |}\Re (\AA U  ,U )_{\HH}=\frac{\la^{-2}}{\kappa_1 -|\kappa_2 |}\Re (F  ,U )_{\HH} \leq \frac{\la^{-2}}{\kappa_1-|\kappa_2 |}\|F \|_{\HH} \|U \|_{\HH}.  \end{equation}Thus, from \eqref{4.12} and the fact that $ \|F\|_{\HH}=o(1)$ and $\|U\|_{\HH}=1$, we obtain \eqref{vnxo}. Now, from  \eqref{eq1ps}, we deduce that
		\begin{equation}\label{3.162}	
		\displaystyle \intdnb |u_x |^2 dx \leq\displaystyle 2\la^{-2} \intdnb |v_x |^2 dx +2\la^{-4}\intdnb |f^{1}_x |^2 dx \leq\displaystyle 2\la^{-2} \intdnb |v_x |^2 dx +2\la^{-4}\intdx |f^{1}_x |^2 dx.\end{equation}
		Therefore, from \eqref{vnxo}, \eqref{3.162} and the fact that $\displaystyle\|f^{1}_x \|_{L^2 (0,L)} =o(1)$, we obtain \eqref{unxo}.
		Next, from \eqref{eq5ps} and the fact that  $\eta (\cdot,0)=v (\cdot)$ , we get 
		\begin{equation}\label{etanps}
		\eta (x,\rho)=ve^{-i\la \tau \rho}+\tau\la^{-2}\int_{0}^{\rho}e^{i\la \tau (s-\rho)}f^{5}(x,s)ds, \quad (x,\rho)\in (0,L )\times (0,1).
		\end{equation}
		From  \eqref{etanps}, we deduce that \begin{equation}\label{3.182}
		\displaystyle\intdnb \int_{0}^{1} |\eta_x (\cdot,\rho)|^2 d\rho dx  \leq 
	 \displaystyle 2\intdnb |v_x   |^2 dx +\tau^{2}\la^{-4}\intdnb \int_{0}^{1} |f^{5}_x (\cdot,s)|^2 ds dx. 
	\end{equation}
		Thus, from \eqref{vnxo}, \eqref{3.182}  and the fact that $f^5(\cdot,\rho)\to 0$ in $\mathcal{W}$, we obtain \eqref{etao}.
		On the other hand,	from  \eqref{etanps}, we have \begin{equation*}
		\eta_{x} (\cdot,1)=v_{x}e^{-i\la \tau}+\tau\la^{-2} \int_{0}^{1}e^{i\la \tau (s-1)}f^{5}_{x}(\cdot,s)ds,
		\end{equation*} consequently, similar to the previous proof, we obtain \eqref{eta1o}.
	Next, it is clear to see that 
	\begin{equation*}\label{itsclear2}
		\intdnb |S_1 (u,v,\eta)|^2 dx =\intdnb |au_x +\kappa_{1}v_x +\kappa_{2} \eta_x (\cdot,1)|^2 dx\leq 3a^2 \intdnb|u_{x}|^2 dx +3\kappa_{1}^2 \intdnb |v_{x}|^2 dx +3\kappa_{2}^2 \intdnb |\eta_{x}(\cdot,1)|^2 dx.
		\end{equation*}
		Finally, from \eqref{vnxo}, \eqref{unxo}, \eqref{eta1o} and the above estimation, we obtain \eqref{So}. The proof is thus complete.
\end{proof}

\begin{figure}[h]
	\begin{center}
		\begin{tikzpicture}
		\draw[->](1,0)--(12,0);
		\draw[->](1,0)--(1,3);
	
		
		\node[black,below] at (1,0){\scalebox{0.75}{$0$}};
		\node at (1,0) [circle, scale=0.3, draw=black!80,fill=black!80] {};
		\node[black,below] at (2,0){\scalebox{0.75}{$\varepsilon$}};
		\node at (2,0) [circle, scale=0.3, draw=black!80,fill=black!80] {};
		\node[black,below] at (3,0){\scalebox{0.75}{$2\varepsilon $}};
		\node at (3,0) [circle, scale=0.3, draw=black!80,fill=black!80] {};

		\node[black,below] at (4,0){\scalebox{0.75}{$\alpha $}};
		\node at (4,0) [circle, scale=0.3, draw=black!80,fill=black!80] {};

		\node[black,below] at (5,0){\scalebox{0.75}{$\alpha +\varepsilon$}};
		\node at (5,0) [circle, scale=0.3, draw=black!80,fill=black!80] {};
		
		\node[black,below] at (6,0){\scalebox{0.75}{$\beta-3\varepsilon$}};
		\node at (6,0) [circle, scale=0.3, draw=black!80,fill=black!80] {};
		
		\node[black,below] at (7,0){\scalebox{0.75}{$\beta-2\varepsilon$}};
		\node at (7,0) [circle, scale=0.3, draw=black!80,fill=black!80] {};
		
		\node[black,below] at (8,0){\scalebox{0.75}{$\beta-\varepsilon$}};
		\node at (8,0) [circle, scale=0.3, draw=black!80,fill=black!80] {};
		
		\node[black,below] at (9,0){\scalebox{0.75}{$\beta$}};
		\node at (9,0) [circle, scale=0.3, draw=black!80,fill=black!80] {};
		
		\node[black,below] at (10,0){\scalebox{0.75}{$\gamma$}};
		\node at (10,0) [circle, scale=0.3, draw=black!80,fill=black!80] {};
		
		\node[black,below] at (11,0){\scalebox{0.75}{$L$}};
		\node at (11,0) [circle, scale=0.3, draw=black!80,fill=black!80] {};

		
		\node at (1,2.022) [circle, scale=0.3, draw=black!80,fill=black!80] {};
		
		\node[black,left] at (1,2.022){\scalebox{0.75}{$1$}};

		\node[black,right] at (12.5,3){\scalebox{0.75}{$\theta_1$}};
		\node[black,right] at (12.5,2.5){\scalebox{0.75}{$\theta_2$}};
		\node[black,right] at (12.5,2 ){\scalebox{0.75}{$\theta_3$}};

		\draw [blue] (1,0.022) to[out=0.40,in=180] (2,2.022) ;
		\draw[-,blue](2,2.022)--(8,2.022);
		\draw [blue] (8,2.022) to[out=0.40,in=180] (9,0.022) ;
		\draw[-,blue](9,0.022)--(11,0.022);
		\draw[-,red](1,0.011)--(2,0.011);
		\draw [red] (2,0.011) to[out=0.40,in=180] (3,2.011) ;
		\draw[-,red](3,2.011)--(7,2.011);
		\draw [red] (7,2.011) to[out=0.40,in=180] (8,0.011) ;
		\draw[-,red](8,0.011)--(11,0.011);
		\draw[-,green](1,0)--(4,0);
		\draw [green] (4,0) to[out=0.40,in=180] (5,2) ;
		\draw[-,green](5,2)--(6,2);
		\draw [green] (6,2) to[out=0.40,in=180] (7,0) ;
		\draw[-,green](7,0)--(11,0);
		
		
		
		\draw[-,blue](12,3)--(12.5,3);
		\draw[-,red](12,2.5)--(12.5,2.5);
		\draw[-,green](12,2)--(12.5,2);
		\end{tikzpicture}\end{center}
	\caption{Geometric description of the functions $\theta_1$, $\theta_2$ and $\theta_3$.}\label{Fig3}
\end{figure}
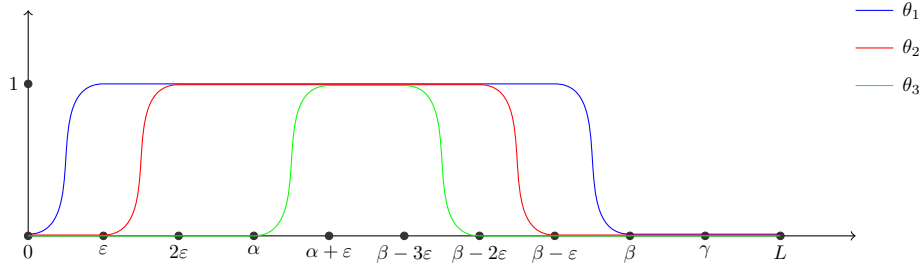

\begin{lem}{\rm\label{intvpslemma}
		Let $0<\varepsilon <\min\left(\frac{\alpha}{2},\frac{\beta -\alpha }{4} \right)$.  Under the  hypothesis \eqref{H}, the solution $U =(u ,v ,y ,z ,\eta (\cdot, \rho) )^{\top}\in D(\AA )$ of  system \eqref{eq1ps}-\eqref{eq5ps} satisfies the following estimation}
	\begin{eqnarray}
		\int_{\varepsilon}^{\beta -\varepsilon} |v|^2 dx =o(1).\label{intvps}
	\end{eqnarray}
\end{lem}
\begin{proof}
	First, we fix a cut-off function $\theta_1 \in C^1 ([0,L])$ (see Figure \ref{Fig3}) such that $0\leq \theta_1 (x)\leq 1$, for all $x\in[0,L]$ and 
 \begin{equation*}
\theta_1 (x)= 	\left \{ \begin{array}{lll}
1 &\text{if} \quad \,\,  x \in [\varepsilon ,\beta -\varepsilon],&\vspace{0.1cm}\\
0 &\text{if } \quad x \in \{0\}\cup [\beta ,L],&
\end{array}	\right. \qquad\qquad
\end{equation*}
and set
 \begin{equation*}
\max_{x\in [0 , L]}|\theta_1^{\prime}(x)|=M_{\theta_1^{\prime}}.
\end{equation*}
Multiplying \eqref{eq2ps} by $\la^{-1}\theta_1 \overline{v}$, integrating over $(0,L)$, then taking the imaginary part, we obtain
\begin{equation*}\label{eq1cutof1}
	\intdx \theta_1 |v|^2 dx -\Im \left\{\la^{-1}\intdx \theta_1 \left(S_b \left(u,v,\eta\right)\right)_x \overline{v}dx \right\} +\Im \left\{\la^{-1}\intdx c(\cdot)\theta_1 z\overline{v}dx \right\}=\Im \left\{\la^{-3}\intdx \theta_1 f^2 \overline{v}dx \right\}.
\end{equation*}
Using integration by parts in  the above equation and the fact that  $v(0)=v(L)=0$, we get
\begin{equation}\label{eq2cutof1}
	\intdx \theta_1 |v|^2 dx= -\Im \left\{\frac{1}{\la}\intdx (\theta_{1}^{\prime}\overline{v}+\theta_1 \overline{v_x } )  S_b(u,v,\eta) dx \right\} -\Im \left\{\frac{1}{\la}\intdx c(\cdot)\theta_1 z\overline{v}dx \right\}+\Im \left\{\frac{1}{\la^3}\intdx \theta_1 f^2 \overline{v}dx \right\}.
\end{equation}
Using  the definition of $c(\cdot) $, $S_b (u,v,\eta)$ and $\theta_1 $, then, using  Cauchy-Schwarz inequality, we obtain
\begin{equation*}\label{eq3cutof1}
\begin{array}{lll}
\displaystyle \left|\Im \left\{\la^{-1}\intdx (\theta_{1}^{\prime}\overline{v}+\theta_1 \overline{v_x } )  S_b (u,v,\eta) dx\right\}\right|=\left|\Im \left\{\la^{-1}\intdnb (\theta_{1}^{\prime}\overline{v}+\theta_1 \overline{v_x } )  S_1 (u,v,\eta) dx\right\}\right|\vspace{0.25cm}\\\hspace{1cm}\displaystyle \leq |\la|^{-1}\left[M_{\theta_1^{\prime}}\left(\intdnb|v|^2 dx \right)^{\frac{1}{2}}+\left(\intdnb |v_x|^2dx\right)^{\frac{1}{2}}\right]\left(\intdnb|S_1 (u,v,\eta)|^2dx\right)^{\frac{1}{2}}
\end{array}
\end{equation*} and 
\begin{equation*}\label{eq4cutof1}
\left| \Im \left\{\la^{-1}\intdx c(\cdot)\theta_1 z\overline{v}dx \right\}\right|=\left|\Im \left\{c_0 \la^{-1} \int_{\alpha }^{\beta}\theta_1 z\overline{v}dx \right\}\right|\leq c_0 |\la|^{-1}\left(\intdni|z|^2dx\right)^{\frac{1}{2}}\left(\intdni|v|^2dx\right)^{\frac{1}{2}}.
\end{equation*}
Thus, from the above inequalities, Lemma \ref{firstlemmaps} and the fact that $v$ and $z$ are uniformly bounded in $L^2 (0,L)$, we obtain 
\begin{equation}\label{2terms1}
-\Im \left\{\la^{-1}\intdx (\theta_{1}^{\prime}\overline{v}+\theta_1 \overline{v_x } )  S_b(u,v,\eta) dx\right\}=o(\la^{-2}) \ \  \text{and}\ \   -\Im \left\{\la^{-1}\intdx c(\cdot)\theta_1 z\overline{v}dx \right\}=O(|\la|^{-1})= o(1).
\end{equation}Inserting \eqref{2terms1} in \eqref{eq2cutof1}, then using the fact that $v$ is uniformly bounded in $L^2 (0,L)$ and $\displaystyle\|f^2 \|_{L^2 (0,L)}=o(1)$, we obtain 
\begin{equation*}\label{4.24}
	\intdx \theta_1 |v|^2 dx=o(1).
\end{equation*}
Finally, from the above estimation and the definition of $\theta_1 $, we obtain \eqref{intvps}. The proof is thus complete.
\end{proof}
\begin{lem}{\rm\label{intzyxpslemma}
	Let $0<\varepsilon <\min\left(\frac{\alpha}{2},\frac{\beta -\alpha }{4} \right)$.  Under the  hypothesis \eqref{H},  the solution $U =(u ,v ,y ,z ,\eta(\cdot,\rho) )^{\top}\in D(\AA )$ of system \eqref{eq1ps}-\eqref{eq5ps} satisfies the following estimations
	\begin{equation}\label{intzyxps}
	\int_{\alpha }^{\beta -2\varepsilon} |z |^2 dx =o(1) \quad \text{and}\quad
	\int_{\alpha +\varepsilon }^{\beta -3\varepsilon} |y_x |^2 dx =o(1). 
	\end{equation}}
\end{lem}
\begin{proof} First, we fix a cut-off function $\theta_2 \in C^1 ([0,L])$ (see figure \ref{Fig3}) such that $0\leq \theta_2 (x)\leq 1$, for all $x\in[0,L]$ and 
	\begin{equation*}
	\theta_2 (x)= 	\left \{ \begin{array}{lll}
	0 &\text{if } \quad x \in [0,\varepsilon]\cup [\beta -\varepsilon ,L],&\vspace{0.1cm}\\
	1 & \text{if} \quad \,\, x \in [2\varepsilon ,\beta -2\varepsilon],&
	\end{array}	\right.
	\end{equation*} 
	and set 
	\begin{equation*}
	\max_{x\in [0 , L]}|\theta_2^{\prime}(x)|=M_{\theta_2^{\prime}}.
	\end{equation*}
	 Multiplying  \eqref{eq2ps} and \eqref{eq4ps} by $\theta_2 \overline{z}$ and $\theta_2 \overline{v}$ respectively, integrating over $(0,L)$, then taking the real part, we obtain
	\begin{equation}\label{eq1cutof2}
		\Re\left\{i\la \intdx \theta_2 v\overline{z}dx  \right\}-\Re\left  \{\intdx \theta_2 (S_b (u,v,\eta))_x \overline{z}dx \right\}+\intdx c(\cdot)\theta_2 |z|^2 dx  =\Re\left\{\la^{-2}\intdx \theta_2 f^2 \overline{z}dx \right\}\end{equation} and \begin{equation}\label{eq2cutof2}
		\quad	\Re\left\{i\la \intdx \theta_2 z\overline{v}dx  \right\}-\Re\left  \{\intdx \theta_2 y_{xx} \overline{v}dx \right\}-\intdx c(\cdot)\theta_2 |v|^2 dx =\Re\left\{\la^{-2}\intdx \theta_2 f^4 \overline{v}dx \right\}.\end{equation}
Adding  \eqref{eq1cutof2} and \eqref{eq2cutof2}, then using integration by parts and the fact that $v(0)=v(L)=0$ and $z(0)=z(L)=0$, we get 
	\begin{equation}\label{eq3cutof2}
	\begin{array}{lll}
	\displaystyle	\intdx c(\cdot)\theta_2 |z|^2 dx =	\displaystyle\intdx c(\cdot)\theta_2 |v|^2 dx -\Re\left\{\intdx (\theta_{2}^{\prime}\overline{z}+\theta_2 \overline{z_x })S_b(u,v,\eta) dx \right\}\vspace{0.25cm}\\\displaystyle -\,\Re\left\{\intdx (\theta_{2}^{\prime}\overline{v}+\theta_2 \overline{v_x })y_x  dx  \right\} + \Re\left\{\la^{-2}\intdx \theta_2 f^2 \overline{z}dx \right\}+\Re\left\{\la^{-2}\intdx \theta_2 f^4 \overline{v}dx \right\}.
	\end{array}
	\end{equation}From \eqref{eq3ps}, we deduce that 
	\begin{equation}\label{3xps}
	\overline{z_x }=-i\la \overline{y_x } -\la^{-2}\overline{f^{3}_x }.
	\end{equation}
	Using  \eqref{3xps} and the definition of $S_b (u,v,\eta)$ and $\theta_2 $, then using Cauchy-Schwarz inequality, we obtain	
	\begin{equation*}
	\begin{array}{lll}
\displaystyle	\left|\Re\left\{\intdx (\theta^{\prime}_2 \overline{z} +\theta_2 \overline{z_x } )S_b (u,v,\eta)dx \right\}\right|= \displaystyle \left| \Re\left\{\int_{\varepsilon}^{\beta-\varepsilon}\left[\theta_2^{\prime}\overline{z}+\theta_2 (-i\la \overline{y_x}-\la^{-2}\overline{f^3_x})\right]S_1 (u,v,\eta)dx\right\} \right|\vspace{0.25cm}\\\hspace{0.5cm}\displaystyle \leq \left[ M_{\theta_2^{\prime}}\left(\int_{\varepsilon}^{\beta-\varepsilon}|z|^2dx\right)^{\frac{1}{2}}+|\la|\left(\int_{\varepsilon}^{\beta-\varepsilon}|y_x|^2dx\right)^{\frac{1}{2}}+\la^{-2}\left(\int_{\varepsilon}^{\beta-\varepsilon}|f^3_x|^2dx\right)^{\frac{1}{2}} \right]\left(\int_{\varepsilon}^{\beta-\varepsilon}|S_1 (u,v,\eta)|^2 dx\right)^{\frac{1}{2}}
	\end{array}
	\end{equation*}
and 
\begin{equation*}
	\begin{array}{lll}
\displaystyle	 \left|\Re\left\{\intdx (\theta_{2}^{\prime}\overline{v}+\theta_2 \overline{v_x })y_x  dx  \right\}\right|= \left|\Re\left\{\int_{\varepsilon}^{\beta-\varepsilon} (\theta_{2}^{\prime}\overline{v}+\theta_2 \overline{v_x })y_x  dx  \right\}\right|\vspace{0.25cm}\\\hspace{1cm}\displaystyle \leq \left[M_{\theta_2^{\prime}}\left(\int_{\varepsilon}^{\beta-\varepsilon}|v|^2dx\right)^{\frac{1}{2}}+\left(\int_{\varepsilon}^{\beta-\varepsilon}|v_x|^2dx\right)^{\frac{1}{2}} \right]\left(\int_{\varepsilon}^{\beta-\varepsilon}|y_x|^2 dx\right)^{\frac{1}{2}}.
	\end{array}\qquad\quad
\end{equation*}
Thus, from the above inequalities, Lemmas  \ref{firstlemmaps}, \ref{intvpslemma}  and the  fact that  $y_x $, $z $ are uniformly bounded in $L^2 (0,L)$ and $\displaystyle\|f^{3}_x \|_{L^2 (0,L)}=o(1)$, we obtain \begin{equation}\label{4.34}
	-\Re\left\{\intdx (\theta^{\prime}_2 \overline{z} +\theta_2 \overline{z_x } )S_b(u,v,\eta) dx \right\}= o(1) \quad \text{and} \quad 	\displaystyle	-\Re\left\{\intdx (\theta_{2}^{\prime}\overline{v}+\theta_2 \overline{v_x })y_x  dx  \right\}= o(1).
\end{equation}
Inserting \eqref{4.34} in \eqref{eq3cutof2}, then using the fact that $v$, $z$ are uniformly bounded in $L^2 (0,L)$ and $\displaystyle\|f^2 \|_{L^2 (0,L)}=o(1)$, $\displaystyle\|f^4 \|_{L^2 (0,L)}=o(1)$, we obtain 
\begin{equation*}\label{4.33}
	\intdx c(\cdot)\theta_2 |z|^2 dx =	\displaystyle\intdx c(\cdot)\theta_2 |v|^2 dx +o(1).
\end{equation*}
Therefore, from the above estimation, Lemma \ref{intvpslemma} and  the definition of $c(\cdot)$ and $\theta_2 $, we obtain the first estimation in \eqref{intzyxps}. On the other hand, let us fix a cut-off function $\theta_3 \in C^1 ([0,L])$ (see Figure \ref{Fig3}) such that $0\leq \theta_3 (x)\leq1$, for all $x\in [0,L]$ and 
	\begin{equation*}
	\theta_3 (x)= 	\left \{ \begin{array}{lll}
		0 &\text{if } \quad x \in [0,\alpha]\cup [\beta -2\varepsilon ,L],&\vspace{0.1cm}\\
	1 & \text{if} \quad \,\, x \in [\alpha +\varepsilon ,\beta -3\varepsilon],&
	\end{array}	\right.\qquad
	\end{equation*}
	Now, multiplying  \eqref{eq4ps} by $-\la^{-1}\theta_3 \overline{z}$, integrating over $(0,L)$, then taking the imaginary part, we obtain
\begin{equation*}\label{eq1cutoff3}	-\intdx \theta_3 |z|^2 dx +\Im\left\{\la^{-1}\intdx \theta_3 y_{xx}\overline{z}dx \right\}+\Im\left\{\la^{-1}\intdx c(\cdot)\theta_3 v\overline{z}dx \right\}= -\Im \left\{\la^{-3}\intdx \theta_{3}f^4 \overline{z}dx \right\}.
	\end{equation*}
	Using  integration by parts in the above equation and the fact that $z(0)=z(L)=0$, then using \eqref{3xps}, we get 
	\begin{equation}\label{4.37}
	\begin{array}{lll}
\displaystyle	\intdx \theta_3 |y_x |^2 dx =\displaystyle\intdx \theta_3 |z|^2 dx +\Im\left\{  \la^{-1}\intdx \theta_{3}^{\prime} y_x \overline{z}dx \right\}-\Im\left\{\la^{-1}\intdx c(\cdot)\theta_3 v\overline{z}dx \right\}\vspace{0.25cm}\\\hspace{2.5cm}\displaystyle -\, \Im \left\{\la^{-3}\intdx \theta_{3} \overline{f^{3}_x }y_x dx \right\}-\Im \left\{\la^{-3}\intdx \theta_{3}f^4 \overline{z}dx \right\}.
	\end{array}
	\end{equation}From the definition of $c(\cdot)$ and $\theta_3 $, the first estimation of \eqref{intzyxps} and the fact that $v$ and $y_x $ are uniformly bounded in $L^2 (0,L)$, we obtain
	\begin{equation}\label{4.38}
\left\{	\begin{array}{lll}
\displaystyle	 \Im\left\{  \la^{-1}\intdx \theta_{3}^{\prime} y_x \overline{z}dx \right\}= \Im\left\{  \la^{-1}\int_{\alpha}^{\beta-2\varepsilon} \theta_{3}^{\prime} y_x \overline{z}dx \right\}=o\left(|\la|^{-1}\right),\vspace{0.25cm}\\\displaystyle  -\Im\left\{\la^{-1}\intdx c(\cdot)\theta_3 v\overline{z}dx \right\}=-\Im\left\{c_0 \la^{-1}\int_{\alpha}^{\beta-2\varepsilon} \theta_3 v\overline{z}dx \right\}= o\left(|\la|^{-1}\right).
	 \end{array}
	 \right.
	\end{equation}
	Inserting \eqref{4.38} in \eqref{4.37}, then using the fact that $y_x $, $z$ are uniformly bounded in $L^2 (0,L)$ and $\displaystyle \|f^{3}_x \|_{L^2 (0,L)}=o(1)$, $\displaystyle \|f^{4} \|_{L^2 (0,L)}=o(1)$, we get   \begin{equation*}\label{4.38ps}
	\intdx \theta_3 |y_x |^2 dx =\displaystyle\intdx \theta_3 |z|^2 dx +o(|\la|^{-1}).
	\end{equation*}
	Finally, from the above estimation, the first estimation of \eqref{intzyxps} and the definition of $\theta_3 $, we obtain the second estimation in \eqref{intzyxps}. The proof is thus complete.
\end{proof}
\begin{lem}\label{boundariesps}{\rm$0<\varepsilon <\min\left(\frac{\alpha}{2},\frac{\beta -\alpha }{4} \right)$.  Under the  hypothesis \eqref{H},  the solution $U =(u ,v ,y ,z ,\eta (\cdot, \rho))^{\top}\in D(\AA )$ of  system \eqref{eq1ps}-\eqref{eq5ps} satisfies the following estimations
	\begin{eqnarray}
	|v(\gamma )|^2 +|v(\beta -3\varepsilon)|^2 +a |u_x (\gamma )|^2+a^{-1}|\left(S_1(u,v,\eta)\right) (\beta -3\varepsilon)|^2 =O(1),\label{vbps}\\
|z(\gamma )|^2 +|z(\beta -3\varepsilon)|^2 +|y_x (\gamma )|^2 +|y_x (\beta -3\varepsilon)|^2 =O(1).\label{zbps}
	\end{eqnarray}}
\end{lem}
\begin{proof}
	First, we fix a function $g_2 \in C^1 ([\beta -3\varepsilon ,\gamma])$ such that 
	\begin{equation*}
		g_2 (\beta -3\varepsilon )=-g_2 (\gamma )=1  \quad  \text{and \ set }\quad \max_{x\in [\beta -3\varepsilon ,\gamma]}|g_2 (x)|=M_{g_2 } \quad \text{and}\quad \max_{x\in [\beta -3\varepsilon ,\gamma]}|g_{2}^{\prime} (x)|=M_{g_{2}^{\prime} }.
	\end{equation*}
	 From  \eqref{eq1ps}, we deduce that  \begin{equation}\label{1xps}
	i\la u_x -v_x =\la^{-2}f^{1}_x .
	\end{equation}Multiplying  \eqref{1xps} and \eqref{eq2ps} by $2g_2 \overline{v}$ and $2a^{-1}g_2 \overline{S_b }(u,v,\eta)$ respectively, integrating over $(\beta -3\varepsilon ,\gamma)$, using the definition of  $c(\cdot)$ and $S_b (u,v,\eta)$, then taking the real part, we obtain
	\begin{equation*}\label{intbvps}
		\Re\left\{2i\la\int_{\beta -3\varepsilon}^{\gamma }g_2 u_x \overline{v}dx \right\}-\int_{\beta -3\varepsilon}^{\gamma}g_2 \left(\left|v\right|^2 \right)_x dx =\Re\left\{2\la^{-2}\int_{\beta -3\varepsilon}^{\gamma }g_2 f^{1}_x \overline{v} dx\right\}
	\end{equation*} 
	and 
	\begin{equation*}\label{S_1vbou}
 \begin{array}{lll}
 &&\displaystyle	\Re\left\{2i\la \int_{\beta -3\varepsilon}^{\gamma } g_2 v \overline{u_x }dx \right\}+	\Re\left\{\frac{2i\la}{a} \int_{\beta -3\varepsilon}^{\beta } g_2 v \left(\kappa_1 \overline{v_x }+\kappa_{2} \overline{\eta_x }(\cdot,1) \right)dx \right\}-a^{-1} \int_{\beta -3\varepsilon}^{\beta }g_2 \left(\left|S_1(u,v,\eta) \right|^2 \right)_x dx \vspace{0.25cm}\\&&\displaystyle -\,a \int_{\beta }^{\gamma }g_2 \left(\left|u_x \right|^2 \right)_x dx +\Re\left\{\frac{2c_0}{a} \int_{\beta -3\varepsilon}^{\beta }g_2 z \overline{S_1 }(u,v,\eta)dx\right\} +\Re\left\{	2c_0 \int_{\beta }^{\gamma }g_2 z \overline{u_x  }dx \right\}\vspace{0.25cm}\\&&=\displaystyle \Re\left\{\frac{2}{a\la^2}\int_{\beta -3\varepsilon}^{\beta }g_2 f^2 \overline{S_1 }(u,v,\eta)dx\right\} +\Re\left\{\frac{2}{\la^2}\int_{\beta}^{\gamma} g_2 f^2 \overline{u_x }dx  \right\}.
	\end{array} 
	\end{equation*}
	Adding  the above Equations, then using integration by parts, we get 
\begin{equation*}
	\begin{array}{lll}
&&\displaystyle	\left[-g_2 \left|v\right|^2 \right]_{\beta -3\varepsilon}^{\gamma} +\left[-a^{-1}g_2 \left|S_1(u,v,\eta) \right|^2 \right]_{\beta -3\varepsilon}^{\beta} +\left[-ag_2 \left|u_x \right|^2 \right]_{\beta }^{\gamma}=\displaystyle-\,\int_{\beta -3\varepsilon}^{\gamma }g_{2}^{\prime}|v|^2 dx -a^{-1}\int_{\beta -3\varepsilon}^{\beta }g_{2}^{\prime}|S_1 (u,v,\eta)|^2 dx  \vspace{0.25cm}\\&&\displaystyle  -\,a\int_{\beta }^{\gamma }g_{2}^{\prime}|u_x |^2 dx-\Re\left\{\frac{2i\la}{ a} \int_{\beta -3\varepsilon}^{\beta } g_2 v \left(\kappa_1 \overline{v_x }+\kappa_{2} \overline{\eta_x }(\cdot,1) \right)dx \right\}- \Re\left\{\frac{2c_0}{ a}\int_{\beta -3\varepsilon}^{\beta }g_2 z \overline{S_1 }(u,v,\eta)dx\right\}  \vspace{0.25cm}\\&&\displaystyle -\,\Re\left\{	2c_0 \int_{\beta }^{\gamma }g_2 z \overline{u_x  }dx \right\}+ \Re\left\{\frac{2}{\la^{2}}\int_{\beta -3\varepsilon}^{\gamma }g_2 f^{1}_x \overline{v}dx \right\}+ \Re\left\{\frac{2}{a\la^{2}}\int_{\beta -3\varepsilon}^{\beta }g_2 f^2 \overline{S_1 }(u,v,\eta)dx\right\} +\Re\left\{\frac{2}{\la^2}\int_{\beta}^{\gamma} g_2 f^2 \overline{u_x }dx  \right\}.
	\end{array}
\end{equation*}
Using the definition of $g_2 $ and Cauchy-Schwarz inequality in the above Equation, we obtain
\begin{equation*}\label{4.44}
\begin{array}{lll}
\hspace{0.5cm}\displaystyle |v(\gamma )|^2 +|v(\beta -3\varepsilon)|^2 +a |u_x (\gamma )|^2+a^{-1}\left|\left(S_1 (u,v,\eta)\right)\left(\beta -3\varepsilon\right)\right|^2  +\mathcal{K}(\beta)\vspace{0.25cm}\\\hspace{1cm} \leq\displaystyle M_{g_{2}^{\prime}}\left[\int_{\beta-3\varepsilon}^{\gamma} |v|^2 dx +a^{-1}\int_{\beta -3\varepsilon}^{\beta }|S_1 (u,v,\eta)|^2 dx +a \int_{\beta}^{\gamma} |u_x |^2 dx \right]\vspace{0.25cm}\\
\hspace{1.25cm}\displaystyle+\,\frac{2|\la|M_{g_2 }}{a}\left[\kappa_{1} \left(\int_{\beta -3\varepsilon}^{\beta }|v_x |^2 dx \right)^{\frac{1}{2}} +|\kappa_{2}| \left(\int_{\beta -3\varepsilon}^{\beta }|\eta_x (\cdot,1) |^2 dx \right)^{\frac{1}{2}}\right]\left(\int_{\beta-3\varepsilon}^{\beta} |v|^2 dx \right)^{\frac{1}{2}}\vspace{0.25cm}\\
\hspace{1.25cm}\displaystyle +\,\frac{2c_0 M_{g_2 }}{a} \left(\int_{\beta -3\varepsilon}^{\beta }|S_1 (u,v,\eta) |^2 dx \right)^{\frac{1}{2}}\left(\int_{\beta-3\varepsilon}^{\beta}|z|^2dx\right)^{\frac{1}{2}} +2c_0 M_{g_2 }\left(\int_{\beta}^{\gamma} |z |^2 dx \right)^{\frac{1}{2}}\left(\int_{\beta}^{\gamma}|u_x|^2 dx\right)^{\frac{1}{2}}\vspace{0.25cm}\\
\hspace{1.25cm}\displaystyle +\, \frac{2M_{g_2 }}{\la^2}\left(\int_{\beta-3\varepsilon}^{\gamma} |f^{1}_x |^2 dx\right)^{\frac{1}{2}}\left(\int_{\beta-3\varepsilon}^{\gamma} |v |^2 dx\right)^{\frac{1}{2}}+\frac{2M_{g_2 }}{a\la^2}\left(\int_{\beta-3\varepsilon}^{\beta} |f^{2} |^2 dx\right)^{\frac{1}{2}}\left(\int_{\beta -3\varepsilon }^{\beta } |S_1 (u,v,\eta)|^2 dx\right)^{\frac{1}{2}}\vspace{0.25cm}\\
\hspace{1.25cm}\displaystyle+\, \frac{2M_{g_2 }}{\la^{2}}\left(\int_{\beta}^{\gamma} |f^2 |^2 dx \right)^{\frac{1}{2}}\left(\int_{\beta}^{\gamma} |u_x |^2 dx \right)^{\frac{1}{2}}.
\end{array}\qquad
\end{equation*}
where $\
\mathcal{K}(\beta)=g_2 (\beta )\left(a|u_x (\beta^{+})|^2 -a^{-1}|\left(S_1 (u,v,\eta)\right)(\beta^{-})|^2 \right)$.
Moreover, since $S_b(u,v,\eta) \in H^1 (0,L)\subset C([0,L])$, then we obtain 
	\begin{equation}\label{contsb}
		|\left(S_1 (u,v,\eta)\right)(\beta^{-})|^2 =|au_x (\beta^{+})|^2 \ \ \text{and consequently}\ \ \mathcal{K}(\beta)=0.
	\end{equation}
	Inserting \eqref{contsb} in the above inequality, then using Lemma \ref{firstlemmaps} and the fact that $u_x $, $v$, $z$ are uniformly bounded in $L^2 (0,L)$ and $\displaystyle\|f^{1}_x \|_{L^2 (0,L)}=o(1)$, $\displaystyle\|f^2 \|_{L^2 (0,L)}=o(1)$, we obtain \eqref{vbps}.
Next, from  \eqref{eq3ps}, we deduce that
\begin{equation}\label{4.46}
	i\la y_x -z_x =\la^{-2}f^3_x .
\end{equation} 
	 Multiplying Equations \eqref{4.46} and \eqref{eq4ps} by $2g_2 \overline{z}$ and $2g_2 \overline{y_x }$ respectively, integrating over $(\beta-3\varepsilon ,\gamma )$, using the definition of $c(\cdot)$, then taking the real part, we obtain 
\begin{equation}\label{intbzbyxps}
		\Re\left\{2i\la \int_{\beta -3\varepsilon}^{\gamma }g_2 y_x \overline{z}dx \right\}-\int_{\beta -3\varepsilon}^{\gamma }g_2\left(\left|z\right|^2 \right)_x dx =\Re\left\{ 2\la^{-2}\int_{\beta -3\varepsilon}^{\gamma }g_2 f^{3}_x \overline{z}dx \right\}
		\end{equation}
		and
	\begin{equation}\label{intbzbyxps1}
	\Re\left\{2i\la \int_{\beta -3\varepsilon}^{\gamma }g_2 z \overline{y_x }dx \right\}-\int_{\beta -3\varepsilon}^{\gamma }g_2 \left(\left|y_x \right|^2 \right)_x dx -\Re\left\{2c_0  \int_{\beta -3\varepsilon}^{\gamma }g_2 v \overline{y_x }dx \right\} =\Re\left\{ 2\la^{-2}\int_{\beta -3\varepsilon}^{\gamma }g_2 f^{4} \overline{y_x }dx \right\}.
\end{equation}
Adding Equations \eqref{intbzbyxps} and \eqref{intbzbyxps1}, then using integration by parts, we obtain 
\begin{equation*}\label{intbzbyxps3}
	\begin{array}{lll}
\displaystyle	\left[ -g_2 \left(\left|z\right|^2 +\left|y_x \right|^2 \right) \right]_{\beta -3\varepsilon}^{\gamma } =\displaystyle -\,\int_{\beta -3\varepsilon}^{\gamma }g_{2}^{\prime}(|z|^2 +|y_x |^2 )dx +\Re\left\{2c_0 \int_{\beta -3\varepsilon}^{\gamma }gv\overline{y_x }dx \right\}+\Re\left\{ 2\la^{-2}\int_{\beta -3\varepsilon}^{\gamma }g_2 f^{3}_x \overline{z}dx \right\}\vspace{0.25cm}\\\hspace{4cm}\displaystyle+\,\Re\left\{ 2\la^{-2}\int_{\beta -3\varepsilon}^{\gamma }g_2 f^{4} \overline{y_x}dx \right\}.
	\end{array}
\end{equation*}
Using the definition of $g_2 $ and Cauchy-Schwarz inequality in the above Equation, we obtain 
\begin{equation*}\label{3.132}
\begin{array}{lll}
\displaystyle|z(\gamma )|^2 +|z(\beta -3\varepsilon)|^2 +|y_x (\gamma )|^2 +|y_x (\beta -3\varepsilon)|^2 \vspace{0.25cm}\\\hspace{1cm} \leq\displaystyle M_{g_{2}^{\prime}}\int_{\beta-3\varepsilon}^{\gamma} \left(\left|z\right|^2 +\left|y_x \right|^2 \right)dx+ 2c_0 M_{g_2 }\left(\int_{\beta-3\varepsilon}^{\gamma} |v|^2 dx\right)^{\frac{1}{2}} \left(\int_{\beta-3\varepsilon}^{\gamma} |y_x |^2 dx\right)^{\frac{1}{2}}\vspace{0.25cm}\\ \hspace{1.25cm}+\, \displaystyle 2\la^{-2} M_{g_2 }\left[\left(\int_{\beta-3\varepsilon}^{\gamma} |f^{3}_x |^2 dx \right)^{\frac{1}{2}}\left(\int_{\beta-3\varepsilon}^{\gamma} |z |^2 dx \right)^{\frac{1}{2}}+\left(\int_{\beta-3\varepsilon}^{\gamma}|f^{4} |^2 dx \right)^{\frac{1}{2}}\left(\int_{\beta-3\varepsilon}^{\gamma} |y_x |^2 dx \right)^{\frac{1}{2}}\right].
\end{array}
\end{equation*}
Finally, from the above inequality, the fact that $v$, $y_x$, $z$ are uniformly bounded in $L^2 (0,L)$ and $\displaystyle\|f^{3}_x \|_{L^2 (0,L)}=o(1)$, $\displaystyle\|f^4 \|_{L^2 (0,L)}=o(1)$, we obtain \eqref{zbps}. The proof is thus complete.
\end{proof}
\begin{lem}\label{2hsn2hynxps}{\rm Let  $h_2  \in C^1 ([0,L])$ be a function.  Under the  hypothesis \eqref{H},  the solution $U =(u ,v ,y ,z ,\eta (\cdot,\rho))^{\top}\in D(\AA)$ of  system \eqref{eq1ps}-\eqref{eq5ps} satisfies the following estimation}
		\begin{equation*}\label{2hsn2hynxeqps}
		\begin{array}{lll}
		&& \displaystyle\intdx h^{\prime}_2 \left(a^{-1}|S_b(u,v,\eta)  |^2 +|v|^2 +|z |^2 + \left|y_x \right|^2\right)dx-\left[h_2 \left(a^{-1}|S_b(u,v,\eta) |^2 +\left|y_x \right|^2 \right)\right]_{0}^{L}\vspace{0.25cm}\\
		&&\displaystyle-\,\Re\left\{2\intdx c(\cdot)h_2 v \overline{y_x  }dx\right\}+\displaystyle\Re\left\{ \frac{2}{a}\intdx c(\cdot)h_2 z \overline{S_b }(u,v,\eta)dx\right\}+\Re \left\{\frac{2i\la}{ a}\intdx b(\cdot)hv^n (\kappa_{1}\overline{v_x }+\kappa_2 \overline{\eta_x }(\cdot,1))dx\right\}\vspace{0.25cm}\\&&=\displaystyle \Re\left\{\frac{2}{\la^{2}}\intdx h_2 \overline{f^{1}_x }v dx \right\}+\Re \left\{\frac{2}{a\la^{2}}\intdx h_2 f^{2}\overline{S_b  }(u,v,\eta)dx \right\}+\Re \left\{\frac{2}{\la^{2}}\intdx h_2 \overline {f^{3}_x } z dx \right\}+\Re \left\{\frac{2}{\la^{2}}\intdx h_2 f^{4}\overline{y_x  }dx \right\}.
		\end{array}\end{equation*}
\end{lem}\begin{proof}
See the proof of Lemma \ref{2hsn2hynx}.
\end{proof}\\\linebreak
Let $0<\varepsilon<\min\left(\frac{\alpha}{2},\frac{\beta-\alpha}{4}\right)$, we fix the cut-off functions $\theta_4 ,\theta_5  \in C^1 ([0,L])$ (see Figure \ref{Fig4}) such that $0\leq \theta_4(x)\leq 1$, $0\leq \theta_5(x)\leq 1$, for all $x\in[0,L]$ and 
\begin{equation*}
	\theta_4 (x)= 	\left \{ \begin{array}{lll}
	1 & \text{if} \quad \,\, x \in [0,\alpha +\varepsilon],&\vspace{0.1cm}\\
	0 &\text{if } \quad x \in [\beta -3\varepsilon,L],&
	\end{array}	\right.\text{and}\quad \theta_5 (x)= 	\left \{ \begin{array}{lll}
	0 & \text{if} \quad \,\, x \in [0,\alpha +\varepsilon],&\vspace{0.1cm}\\
	1 &\text{if } \quad x \in [\beta-3\varepsilon,L],&
	\end{array}	
	\right. 
	\end{equation*}
	\begin{figure}[h]	
		\begin{center}
			\begin{tikzpicture}
			\draw[->](1,0)--(8,0);
			\draw[->](1,0)--(1,3);
			
			
			\node[black,below] at (1,0){\scalebox{0.75}{$0$}};
			\node at (1,0) [circle, scale=0.3, draw=black!80,fill=black!80] {};
			
			\node[black,below] at (2,0){\scalebox{0.75}{$\alpha$}};
			\node at (2,0) [circle, scale=0.3, draw=black!80,fill=black!80] {};
			
			\node[black,below] at (3,0){\scalebox{0.75}{$\alpha +\varepsilon$}};
			\node at (3,0) [circle, scale=0.3, draw=black!80,fill=black!80] {};

			\node[black,below] at (4,0){\scalebox{0.75}{$\beta -3\varepsilon$}};
			\node at (4,0) [circle, scale=0.3, draw=black!80,fill=black!80] {};

			\node[black,below] at (5,0){\scalebox{0.75}{$\beta$}};
			\node at (5,0) [circle, scale=0.3, draw=black!80,fill=black!80] {};
			
			\node[black,below] at (6,0){\scalebox{0.75}{$\gamma$}};
			\node at (6,0) [circle, scale=0.3, draw=black!80,fill=black!80] {};
			
			\node[black,below] at (7,0){\scalebox{0.75}{$L$}};
			\node at (7,0) [circle, scale=0.3, draw=black!80,fill=black!80] {};
			
			
			\node at (1,2) [circle, scale=0.3, draw=black!80,fill=black!80] {};

			\node[black,left] at (1,2){\scalebox{0.75}{$1$}};
			
			\node[black,right] at (8.5,3){\scalebox{0.75}{$\theta_4$}};
			\node[black,right] at (8.5,2.5){\scalebox{0.75}{$\theta_5$}};

			\draw[-,red](1,2)--(3,2);
			\draw [red] (3,2) to[out=0.40,in=180] (4,0) ;
			\draw[-,red](4,0)--(7,0);
			
			\draw[-,blue](1,0)--(3,0);
			\draw [blue] (3,0) to[out=0.40,in=180] (4,2) ;
			\draw[-,blue](4,2)--(7,2);

			\draw[-,red](8,3)--(8.5,3);
			\draw[-,blue](8,2.5)--(8.5,2.5);
	
			\end{tikzpicture}
		\end{center}
		\caption{Geometric description of  the functions $\theta_4$ and $\theta_5$.}\label{Fig4}
	\end{figure}
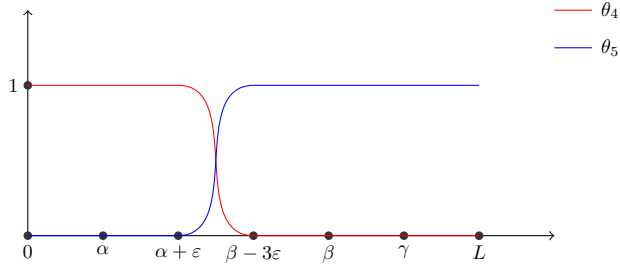
\begin{lem}\label{lastestimation-pol}
\rm{Let $0<\varepsilon <\min\left(\frac{\alpha}{2},\frac{\beta -\alpha }{4} \right)$. Under the  hypothesis \eqref{H},  the solution $U =(u ,v ,y ,z ,\eta (\cdot, \rho) )^{\top}\in D(\AA )$ of the System \eqref{eq1ps}-\eqref{eq5ps} satisfies the following estimations}
 \begin{eqnarray}
 \int_{0}^{\alpha +\varepsilon}|v|^2 dx+ \int_{0}^{\alpha +\varepsilon}|y_x |^2 dx +\int_{0}^{\alpha +\varepsilon} |z|^2 dx =o(1),\label{resulycutoff4}\\
 	a\int_{\beta }^{L }|u_x |^2 dx +\int_{\beta -3\varepsilon }^{L }|v|^2 dx+ \int_{\beta -3\varepsilon}^{L}|y_x |^2 dx +\int_{\beta -3\varepsilon}^{L} |z|^2 dx =o(1)\label{resultcutoff5}.
 \end{eqnarray}
\end{lem}
\begin{proof} 
First, using the result of Lemma \ref{2hsn2hynxps} with $h_2 =x\theta_4 $, we obtain 
	\begin{equation*}\label{eq1cutof4ps}
		\begin{array}{lll}
	 &&	\displaystyle	\int_{0}^{\alpha +\varepsilon}|v|^2 dx+ \int_{0}^{\alpha +\varepsilon}|y_x |^2 dx +\int_{0}^{\alpha +\varepsilon} |z|^2 dx =-\,a^{-1}\int_{0 }^{\alpha +\varepsilon}|S_1(u,v,\eta) |^2 dx \vspace{0.25cm} \\
	 && \displaystyle -\,\int_{\alpha +\varepsilon}^{\beta - 3\varepsilon}\left( \theta_4 +x \theta_4^{\prime}\right)\left(a^{-1}|S_1(u,v,\eta) |^2 +|v|^2 +|y_x |^2 +|z|^2 \right)dx +\Re\left\{2\intdx xc(\cdot)\theta_4 v\overline{y_x }dx \right\}\vspace{0.25cm}\\ 
	  &&-\, \displaystyle \Re\left\{\frac{2}{a}\intdx xc(\cdot)\theta_4 z\overline{S_b }(u,v,\eta)dx \right\} -\Re\left\{\frac{2i\la}{ a}\intdx xb(\cdot)\theta_4 v \left(\kappa_1 \overline{v_x }+\kappa_2 \overline{\eta_x }(.,1)\right) dx  \right\}+ \Re\left\{\frac{2}{\la^{2}}\intdx x\theta_4 \overline{f^{1}_x }v dx \right\} \vspace{0.25cm}\\
	  &&+\, \displaystyle \displaystyle \Re \left\{\frac{2}{a\la^{2}}\intdx x\theta_4f^{2}\overline{S_b  }(u,v,\eta)dx \right\}+\Re \left\{\frac{2}{\la^{2}}\intdx x\theta_4 \overline {f^{3}_x } z dx \right\}+\Re \left\{\frac{2}{\la^{2}}\intdx x\theta_4  f^{4}\overline{y_x  }dx \right\}.
		\end{array}
	\end{equation*}From the above Equation and by using Lemmas \ref{firstlemmaps}, \ref{intvpslemma}, \ref{intzyxpslemma} with the fact that $v$, $y_x $, $z$ are uniformly bounded in $L^2 (0,L)$ and $\displaystyle\|f^{1}_x \|_{L^2 (0,L)}=o(1)$, $\displaystyle\|f^{3}_x \|_{L^2 (0,L)}=o(1)$, $\displaystyle\|f^{4} \|_{L^2 (0,L)}=o(1)$, we obtain 
\begin{equation}\label{4.51}
	\begin{array}{lll}
	\displaystyle	\int_{0}^{\alpha +\varepsilon}|v|^2 dx+ \int_{0}^{\alpha +\varepsilon}|y_x |^2 dx +\int_{0}^{\alpha +\varepsilon} |z|^2 dx =\displaystyle \Re\left\{2\intdx xc(\cdot)\theta_4 v\overline{y_x }dx \right\}\vspace{0.25cm}\\
	\displaystyle -\, \Re\left\{\frac{2}{a}\intdx xc(\cdot)\theta_4 z\overline{S_b }(u,v,\eta)dx \right\}+\Re \left\{\frac{2}{a\la^{2}}\intdx x\theta_4f^{2}\overline{S_b  }(u,v,\eta)dx \right\}\vspace{0.25cm}\\\displaystyle-\,\Re\left\{\frac{2i\la}{ a}\intdx xb(\cdot)\theta_4 v \left(\kappa_1 \overline{v_x }+\kappa_2 \overline{\eta_x }(.,1)\right) dx  \right\}+o(1).
	\end{array}
\end{equation}
Using the definition of $b(\cdot)$, $c(\cdot)$, $S_b(u,v,\eta) $, $\theta_4 $, then using Cauchy-Schwarz inequality, we obtain
\begin{equation*}
\left\{\begin{array}{lll}
\displaystyle \left|\Re\left\{2\intdx xc(\cdot)\theta_4 v\overline{y_x }dx \right\}\right|=\left|\Re\left\{2c_0 \int_{\alpha}^{\beta-3\varepsilon} x\theta_4 v\overline{y_x }dx \right\}\right|\leq 2c_0 (\beta -3\varepsilon)\left(\int_{\alpha }^{\beta -3\varepsilon}|v |^2 dx \right)^{\frac{1}{2}}\left(\int_{\alpha }^{\beta -3\varepsilon} |y_x |^2 dx\right)^{\frac{1}{2}},\vspace{0.25cm}\\
	\displaystyle\left|\Re\left\{\frac{2}{a}\intdx xc(\cdot)\theta_4 z\overline{S_b }(u,v,\eta)dx \right\}\right|=	\displaystyle\left|\Re\left\{\frac{2c_0}{a} \int_{\alpha}^{\beta-3\varepsilon} x\theta_4 z\overline{S_1 }(u,v,\eta)dx \right\}\right|\vspace{0.25cm}\\\hspace{5.25cm}
	\leq \displaystyle \frac{2c_0}{a}(\beta -3\varepsilon)\left(\int_{\alpha}^{\beta-3\varepsilon} |z|^2 dx\right)^{\frac{1}{2}}\left(\int_{\alpha }^{\beta -3\varepsilon}|S_1(u,v,\eta) |^2 dx \right)^{\frac{1}{2}},\vspace{0.25cm}\\
	\displaystyle\left|\Re \left\{\frac{2}{a\la^{2}}\intdx x\theta_4  f^{2}\overline{S_b  }(u,v,\eta)dx \right\}\right|=	\displaystyle\left|\Re \left\{\frac{2}{a\la^{2}}\int_{0}^{\beta-3\varepsilon} x\theta_4  f^{2}\overline{S_1  }(u,v,\eta)dx \right\}\right| \vspace{0.25cm}\\\hspace{5.25cm}
\leq \displaystyle  \frac{2(\beta-3\varepsilon)}{a\la^2 }\left(\int_{0}^{\beta-3\varepsilon}|f^2|^2 dx\right)^{\frac{1}{2}}\left(\int_{0}^{\beta-3\varepsilon}|S_1(u,v,\eta)|^2 dx\right)^{\frac{1}{2}},\vspace{0.25cm}\\
 \displaystyle	\left|\Re\left\{\frac{2i\la} {a}\intdx xb(\cdot)\theta_4 v \left(\kappa_1 \overline{v_x }+\kappa_2 \overline{\eta_x }(.,1)\right) dx  \right\}\right|=	\left|\Re\left\{\frac{2i\la}{a}\int_{0}^{\beta-3\varepsilon} x\theta_4 v \left(\kappa_1 \overline{v_x }+\kappa_2 \overline{\eta_x }(.,1)\right) dx  \right\}\right|\vspace{0.25cm}\\\hspace{3.5cm}
 \displaystyle \leq 
 \frac{2|\la|(\beta-3\varepsilon)}{a}\left[\kappa_{1}\left(\int_{0}^{\beta-3\varepsilon}|v_x|^2 dx\right)^{\frac{1}{2}}+|\kappa_2|\left(\int_{0}^{\beta-3\varepsilon}|\eta_x(\cdot,1)|^2 dx\right)^{\frac{1}{2}}\right]\left(\int_{0}^{\beta-3\varepsilon}|v|^2dx\right)^{\frac{1}{2}}.
	\end{array}
	\right.
	\end{equation*}
	Thus, from the above inequalities, Lemmas \ref{firstlemmaps}, \ref{intvpslemma} and the fact that $u_x $, $v$, $y_x $, $z$ are uniformly bounded in $L^2 (0,L)$ and $\|f^2 \|_{L^2 (0,L)}=o(1)$, we obtain   
\begin{equation}\label{4.54}
\begin{array}{lll}
\displaystyle	\Re\left\{2\intdx xc(\cdot)\theta_4 v\overline{y_x }dx \right\} = o(1), \quad
	-\Re\left\{\frac{2}{a}\intdx xc(\cdot)\theta_4 z\overline{S_b }(u,v,\eta)dx \right\}= o(|\la |^{-1}), \vspace{0.25cm}\\
	\displaystyle \Re \left\{\frac{2}{a\la^{2}}\intdx x\theta_4  f^{2}\overline{S_b  }(u,v,\eta)dx \right\}= o(\la^{-2}),
\ \ -\Re\left\{\frac{2i\la} {a}\intdx xb(\cdot)\theta_4 v \left(\kappa_1 \overline{v_x }+\kappa_2 \overline{\eta_x }(.,1)\right) dx\right\}= o(1).
\end{array}
\end{equation}
Therefore, by inserting \eqref{4.54} in \eqref{4.51}, we obtain \eqref{resulycutoff4}.
On the other hand, 
using the result of Lemma \ref{2hsn2hynxps} with $h=(x-L)\theta_5 $, then using the definition of $b(\cdot)$, $S_b $, $\theta_5 $ and Lemmas \ref{firstlemmaps}, \ref{intvpslemma}, \ref{intzyxpslemma} with the fact that $u_x $, $v$, $y_x $, $z$ are uniformly bounded in $L^2 (0,L)$ and $\displaystyle\|f^{1}_x \|_{L^2 (0,L)}=o(1)$, $\displaystyle\|f^2 \|_{L^2 (0,L)}=o(1)$, $\displaystyle\|f^{3}_x \|_{L^2 (0,L)}=o(1)$, $\displaystyle\|f^4 \|_{L^2 (0,L)}=o(1)$, we obtain
	\begin{equation}\label{4.58}
	\begin{array}{lll}
&&\displaystyle	a\int_{\beta }^{L }|u_x |^2 dx +\int_{\beta -3\varepsilon }^{L }|v|^2 dx+ \int_{\beta -3\varepsilon}^{L}|y_x |^2 dx +\int_{\beta -3\varepsilon}^{L} |z|^2 dx \vspace{0.25cm}\\ &&= \displaystyle  \Re\left\{2\intdx (x-L)c(\cdot)\theta_5 v\overline{y_x }dx  \right\}-\Re\left\{2a^{-1}\intdx (x-L)c(\cdot)\theta_5  z \overline{S_b }dx \right\}  +o(1).
	\end{array}
	\end{equation}Moreover, from the definition of  $c(\cdot)$, $S_b $, $\theta_5 $ and by using  Lemmas \ref{firstlemmaps}, \ref{intvpslemma} with the fact that $y_x $, $z$ are uniformly bounded in $L^2 (0,L)$, we obtain 
	\begin{equation}\label{4.59}
		\begin{array}{lll}
	&&\displaystyle	\Re\left\{2\intdx (x-L)c(\cdot)\theta_5 v\overline{y_x }dx  \right\}-\Re\left\{2a^{-1}\intdx (x-L)c(\cdot)\theta_5  z \overline{S_b }dx \right\}  \vspace{0.25cm}\\ &&=\displaystyle \Re \left\{ 2c_0 \int_{\beta -3\varepsilon}^{\gamma }(x-L)v\overline{y_x } dx \right\}-\Re \left\{ 2c_0 \int_{\beta -3\varepsilon}^{\gamma }(x-L)z\overline{u_x } dx \right\}+o(1).
		\end{array}
	\end{equation}From  \eqref{eq1ps} and \eqref{eq3ps}, we deduce that 
	\begin{equation}
		\overline{u_x }= i\la^{-1}\overline{v_x } +i\la^{-3}\overline{f^{1}_x }\quad \text{and}\quad
			\overline{y_x }=i\la^{-1}\overline{z_x } +i\la^{-3}\overline{f^{3}_x }.\label{4.61}
	\end{equation}
	Substituting  \eqref{4.61} in the right hand side of  \eqref{4.59}, then using  the fact that $v$, $z$ are uniformly bounded in $L^2 (0,L)$ and $\displaystyle\|f^{1}_x \|_{L^2 (0,L)}=o(1)$, $\displaystyle\|f^{3}_x \|_{L^2 (0,L)}=o(1)$, we obtain 
	\begin{equation*}\label{4.62}
		\begin{array}{lll}
		 &&\displaystyle \Re\left\{2\intdx (x-L)c(\cdot)\theta_5 v\overline{y_x }dx  \right\}-\Re\left\{2a^{-1}\intdx (x-L)c(\cdot)\theta_5  z \overline{S_b }dx \right\} \vspace{0.25cm}\\ &&= \displaystyle \Re \left\{ \frac{2c_0 i }{\la} \int_{\beta -3\varepsilon}^{\gamma }(x-L)v\overline{z_x } dx \right\}-\Re \left\{ \frac{2c_0 i}{\la}\int_{\beta -3\varepsilon}^{\gamma }(x-L)z\overline{v_x } dx \right\}+o(1).
		\end{array}
	\end{equation*}
	Using integration by parts to the second integral in the right hand side of the above equation, we obtain 
	\begin{equation}\label{4.63}
	\begin{array}{lll}
	&&\displaystyle \Re\left\{2\intdx (x-L)c(\cdot)\theta_5 v\overline{y_x }dx  \right\}-\Re\left\{2a^{-1}\intdx (x-L)c(\cdot)\theta_5  z \overline{S_b }dx \right\} \vspace{0.25cm}\\ & & =\displaystyle \Re \left\{\frac{ 2c_0 i}{\la}\int_{\beta -3\varepsilon}^{\gamma }z\overline{v} dx \right\}-\Re \left\{ \frac{2c_0 i}{ \la} \left[\left(x-L\right)z\overline{v } \right]_{\beta -3\varepsilon}^{\gamma } \right\}+o(1).
	\end{array}
	\end{equation}
	Furthermore, by using Cauchy-Schwarz inequality, we get
	\begin{equation}\label{4.57}
		\left| \Re \left\{\frac{2c_0 i}{\la}\int_{\beta -3\varepsilon}^{\gamma }z\overline{v}dx \right\}\right|\leq 2c_0 |\la|^{-1}\left(\int_{\beta-3\varepsilon}^{\gamma} |z|^2 dx \right)^{\frac{1}{2}}\left(\int_{\beta-3\varepsilon}^{\gamma} |v|^2 dx \right)^{\frac{1}{2}}
	\end{equation}and 
	\begin{equation}\label{4.58psp}
	 \left|\Re\left\{\frac{2c_0 i}{\la}\left[ \left(x-L\right)z\overline{v}\right]_{\beta -3\varepsilon}^{\gamma } \right\}\right|\leq 2c_0 |\la|^{-1}\left[\left(L-\gamma \right)|z(\gamma )||v(\gamma )|+(L-\beta +3\varepsilon)|z(\beta -3\varepsilon)|\left|v(\beta -3\varepsilon)\right| \right].
	\end{equation}  
	From Lemma \ref{boundariesps}, we deduce that \begin{equation}\label{4.58ps}
		|v(\beta -3\varepsilon)|=O(1), \quad 	|v(\gamma )|=O(1), \quad 	|z(\beta -3\varepsilon)|=O(1)\quad \text{and} \quad	|z(\gamma )|=O(1).
	\end{equation}
	Using the fact that $v$, $z$ are uniformly bounded in $L^2 (0,L)$ in \eqref{4.57} and inserting \eqref{4.58ps} in \eqref{4.58psp}, we obtain 
	\begin{equation}\label{4.60}
	\Re \left\{\frac{2c_0 i}{\la}\int_{\beta -3\varepsilon}^{\gamma }z\overline{v}dx \right\}=O\left(|\la|^{-1}\right)=o(1) \quad\text{and}\quad -\Re\left\{\frac{2c_0 i}{\la}\left[ (x-L)z\overline{v}\right]_{\beta -3\varepsilon}^{\gamma } \right\}=O\left(|\la|^{-1}\right)=o(1).
	\end{equation}
	Inserting \eqref{4.60} in \eqref{4.63}, we get 
	\begin{equation}\label{4.59ps}
	\Re\left\{2\intdx (x-L)c(\cdot)\theta_5 v\overline{y_x }dx  \right\}-	\Re\left\{2a^{-1}\intdx (x-L)c(\cdot)\theta_5  z \overline{S_b }dx \right\} = o(1).
	\end{equation}Finally, inserting \eqref{4.59ps} in \eqref{4.58}, we obtain \eqref{resultcutoff5}. The proof is thus complete.
\end{proof}
\\\linebreak
\textbf{Proof of Theorem \ref{polthm}.} The proof of Theorem    is divided into three steps.\\
\textbf{Step 1.} From Lemmas \ref{firstlemmaps}-\ref{intzyxpslemma}, we obtain 
\begin{equation}\label{4.53}
\left\{\begin{array}{lll}
\displaystyle
\int_0^{\beta}|u_x|^2dx=o(\la^{-4}), \ \int_{0}^{\beta}\int_{0}^{1}|\eta_x (\cdot,\rho)|^2 d\rho dx =o(\la^{-2}),\ \int_{\varepsilon}^{\beta-\varepsilon}|v|^2 dx =o(1),\vspace{0.25cm}\\ 
\displaystyle \int_{\alpha}^{\beta-2\varepsilon}|z|^2 dx =o(1)\ \text{and}\  \int_{\alpha+\varepsilon}^{\beta-3\varepsilon}|y_x|^2 dx =o(1).
\end{array}\right.
\end{equation}\\\linebreak
\textbf{Step 2.} From Lemma \ref{lastestimation-pol} and \eqref{4.53}, we deduce that 
\begin{equation*}
\left\{\begin{array}{lll}
\displaystyle \int_{0}^{\varepsilon}|v|^2 dx =o(1), \ \int_{0}^{\alpha+\varepsilon}|y_x|^2 dx =o(1), \ \int_{0}^{\alpha}|z|^2 dx =o(1),\vspace{0.25cm}\\
\displaystyle 
\int_{\beta}^{L}|u_x|^2 dx =o(1), \ \int_{\beta-\varepsilon}^{L}|v|^2 dx =o(1), \ \int_{\beta-3\varepsilon}^{L}|y_x|^2 dx =o(1) \ \text{and}\  \int_{\beta-2\varepsilon}^{L}|z|^2 dx =o(1).
	\end{array}
	\right.
\end{equation*}
According to \textbf{Step 1} and \textbf{Step 2}, we obtain $\|U\|_{\HH}=o(1)$ in $(0,L)$, which contradicts \eqref{contra-pol2}. Thus, \eqref{Condition-2-pol} is holds true. Next, since the conditions \eqref{4.1} and \eqref{Condition-2-pol} are proved, then  according to Theorem \ref{bt}, the proof of Theorem \ref{polthm} is achieved. The proof is thus complete.\xqed{$\square$} 
\section{Conclusion}
 We have studied the stabilization of a one-dimensional coupled wave equations with non smooth localized  viscoelastic damping of Kelvin-Voigt type and localized  time delay. We proved the strong stability of the system by using Arendt-Batty criteria. Finally, we established a polynomial energy decay rate of order $t^{-1}$.
\appendix
\section{Some notions and theorems of stability has been used}
\noindent In order to make this paper more self-contained, we have introduced this short appendix that brings up the notions of stability that we encounter in this work. 
\begin{defi}\label{App-Definition-A.1}{\rm
Assume that $A$ is the generator of $C_0-$semigroup of contractions $\left(e^{tA}\right)_{t\geq0}$ on a Hilbert space $H$. The $C_0-$semigroup $\left(e^{tA}\right)_{t\geq0}$ is said to be 
\begin{enumerate}
\item[$(1)$] Strongly stable if 
$$
\lim_{t\to +\infty} \|e^{tA}x_0\|_H=0,\quad \forall\, x_0\in H.
$$
\item[$(2)$] Exponentially (or uniformly) stable if there exists two positive constants $M$ and $\varepsilon$ such that 
$$
\|e^{tA}x_0\|_{H}\leq Me^{-\varepsilon t}\|x_0\|_{H},\quad \forall\, t>0,\ \forall\, x_0\in H.
$$
\item[$(3)$] Polynomially stable if there exists two positive constants $C$ and $\alpha$ such that 
$$
\|e^{tA}x_0\|_{H}\leq Ct^{-\alpha}\|A x_0\|_{H},\quad \forall\, t>0,\ \forall\, x_0\in D(A).
$$
\xqed{$\square$}
\end{enumerate}}
\end{defi}
\noindent For proving the strong stability of the $C_0$-semigroup $\left(e^{tA}\right)_{t\geq0}$, we will recall the result obtained by Arendt and Batty in \cite{Arendt01}. 
\begin{Theorem}[Arendt and Batty in \cite{Arendt01}]\label{App-Theorem-A.2}{\rm
{Assume that $A$ is the generator of a C$_0-$semigroup of contractions $\left(e^{tA}\right)_{t\geq0}$  on a Hilbert space $H$. If $A$ has no pure imaginary eigenvalues and  $\sigma\left(A\right)\cap i\mathbb{R}$ is countable,
where $\sigma\left(A\right)$ denotes the spectrum of $A$, then the $C_0$-semigroup $\left(e^{tA}\right)_{t\geq0}$  is strongly stable.}\xqed{$\square$}}
\end{Theorem}
\noindent There exist a second classical method based on Arendt and Batty theorem and the contradiction argument  (see  page 25 in \cite{LiuZheng01}).
\begin{Remark}\label{App-Lemma-A.3}{\rm
 Assume that the unbounded linear operator $A:D(A)\subset H \longmapsto H$ is the generator of a C$_0-$semigroup of contractions $\left(e^{tA}\right)_{t\geq0}$  on a Hilbert space $H$ and suppose that $0\in \rho(A).$ According  to  (page 25 in \cite{LiuZheng01}), in order to prove that \begin{equation}\label{A1}
 \displaystyle	i\R \equiv \left\{ i\la\  | \ \ \la\in\R \right\} \subseteq \rho(A),
\end{equation} we need the following steps:
\begin{enumerate}
\item[(i)] It follows from the fact that $0\in \rho (A)$ and the contraction mapping theorem that for any real number $\la$ with $|\la|<\|A^{-1} \|^{-1} $, the operator $i\la I-A =A(i\la A^{-1} -I)$ is invertible. Furthermore, $\|(i\la I-A)^{-1} \|$ is a continuous function of $\la $ in the interval $\left(-\|A^{-1} \|^{-1} ,\|A^{-1} \|^{-1} \right) $.\\
\item[(ii)] If $\sup \left\{\|(i\la I-A)^{-1}\|\ | \ |\la|<\|A^{-1}\|^{-1} \right\}=M < \infty$, then by the contraction mapping theorem, the operator $i\la I-A =(i\la_0 I-A)(I+i(\la-\la_0 )(i\la_0 I-A)^{-1} )$ with $|\la_0 |<\|A^{-1}\|^{-1}$ is invertible for $|\la -\la_0 | <M^{-1}$. It turns out that by choosing $|\la_0 | $ as close to $\|A^{-1}\|^{-1}$ as we can, we conclude that $\left\{ \la \ | \ |\la|<\|A^{-1}\|^{-1}+M^{-1}\right \} \subset \rho (A)$ and $\|(i\la I -A)^{-1}\|$ is a continuous function of $\la $  in the interval $\left(-\|A^{-1}\|^{-1}-M^{-1},\|A^{-1}\|^{-1}+M^{-1} \right).$\\ 
\item[(iii)] Thus it follows from the argument in (ii) that if \eqref{A1} is false, then there is $\omega \in\R $ with $\|A^{-1}\|^{-1}\leq |\omega| <\infty $ such that $\left\{i\la \ | \ |\la |<|\omega| \right\} \subset \rho(A)  $ and $\sup \left\{\|(i\la -A)^{-1}\| \ | \ |\la |<|\omega | \right\}=\infty$. It turns out that there exists a sequence $\left\{(\lambda_n,{U}_n)\right\}_{n\geq 1}\subset \mathbb{R}\times D\left(A\right),$ with $\lambda_n \to  \omega$ as $n\to\infty,$ $|\lambda_n|<|\omega|$ and $\left\|{U}_n\right\|_{H} = 1$, such that
\begin{equation*}
(i \lambda_n I - A){U}_n ={F}_n \to  0 \ \textrm{in} \ {H},\qquad\text{as }n\to\infty.
\end{equation*}
Then, we will prove \eqref{A1}  by finding a contradiction with $\left\|{U}_n\right\|_{H} = 1$ such as $\left\| {U}_n\right\|_{H}  \to 0 .$\xqed{$\square$}.\\ \end{enumerate}}
\end{Remark}

\noindent Concerning the characterization of polynomial stability stability of a $C_0-$semigroup of contraction $\left(e^{tA}\right)_{t\geq 0}$, we rely on the following result due to Borichev and Tomilov \cite{Borichev01} (see also \cite{Batty01} and \cite{RaoLiu01}).
\begin{Theorem}\label{bt}{\rm
	Assume that $A$ is the generator of a strongly continuous semigroup of contractions $\left(e^{tA}\right)_{t\geq0}$  on $\mathcal{H}$.   If   $ i\mathbb{R}\subset \rho(\mathcal{A})$, then for a fixed $\ell>0$ the following conditions are equivalent}
	\begin{equation}\label{h1}
	\sup_{\lambda\in\mathbb{R}}\left\|\left(i\lambda I-\mathcal{A}\right)^{-1}\right\|_{\mathcal{L}\left(\mathcal{H}\right)}=O\left(|\lambda|^\ell\right),
	\end{equation}
	\begin{equation}\label{h2}
	\|e^{t\mathcal{A}}U_{0}\|^2_{\HH} \leq \frac{C}{t^{\frac{2}{\ell}}}\|U_0\|^2_{D(\AA)},\hspace{0.1cm}\forall t>0,\hspace{0.1cm} U_0\in D(\AA),\hspace{0.1cm} \text{for some}\hspace{0.1cm} C>0.
	\end{equation}\xqed{$\square$}
\end{Theorem}
 
\section*{Acknowledgments}
\noindent The authors thanks professor Serge Nicaise for his valuable discussions and comments.\\ \\
Mohammad Akil would like to thank the Lebanese University for its support.\\ \\
Haidar Badawi would like to thank the LAMAV laboratory of Mathematics of the Universit\'e polytechnique  Hauts-De-France Valenciennes for its support. \\ \\ 
Ali Wehbe would like to thank the CNRS for its support.

\end{document}